\documentclass{amsart}
\usepackage{amssymb,color}

\usepackage[all]{xy}

\newtheorem{theorem}{Theorem}[section]
\newtheorem{claim}[theorem]{Claim}

\newtheorem{example}[theorem]{Example}
\newtheorem{proposition}[theorem]{Proposition}
\newtheorem{corollary}[theorem]{Corollary}
\newtheorem{lemma}[theorem]{Lemma}

\theoremstyle{definition}
\newtheorem{definition}[theorem]{Definition}
\newtheorem{problem}[theorem]{Problem}

\theoremstyle{definition}
\newtheorem{remark}[theorem]{Remark}

\newcommand{\IN}{\mathbb N}
\newcommand{\IR}{\mathbb R}

\newcommand{\IC}{\mathbb C}
\newcommand{\IT}{\mathbb T}
\newcommand{\IZ}{\mathbb Z}

\newcommand{\F}{\mathcal F}
\newcommand{\U}{\mathcal U}
\newcommand{\C}{\mathcal C}
\newcommand{\V}{\mathcal V}
\newcommand{\Iso}{\mathrm{Iso}}
\newcommand{\e}{\varepsilon}
\newcommand{\w}{\omega}
\newcommand{\pr}{\mathrm{pr}}

\newcommand{\Ra}{\Rightarrow}

\newcommand{\Tau}{\mathcal T}

\newcommand{\id}{\mathrm{id}}

\newcommand{\TG}{\mathsf{T\!G}}

\newcommand{\PG}{\mathsf{P\!G}}
\newcommand{\pTS}{\mathsf p\!\mathsf T\!\mathsf S}

\newcommand{\Sym}{\mathrm{Sym}}
\newcommand{\TS}{\mathsf{T\!\,S}}
\newcommand{\pTG}{\mathsf{p\!T\!G}}
\newcommand{\qTG}{\mathsf{q\!T\!G}}
\newcommand{\sTG}{\mathsf{s\!T\!G}}
\newcommand{\rTG}{\mathsf{r\!T\!G}}
\newcommand{\rTS}{\mathsf{r\!T\!S}}
\newcommand{\sTS}{\mathsf{s\!T\!S}}
\newcommand{\TSone}{\mathsf{T_{\!\!1\!}S}}
\newcommand{\K}{\mathcal K}
\newcommand{\dom}{\mathrm{dom}}

\newcommand{\eC}{\mathsf{e}{:}\C}
\newcommand{\iC}{\mathsf{i}{:}\C}
\newcommand{\hC}{\mathsf{h}{:}\C}
\newcommand{\pC}{\mathsf{p}{:}\C}
\newcommand{\cC}{\mathsf{c}{:}\C}

\newcommand{\iPG}{\mathsf{i}{:}\mathsf{P\!G}}
\newcommand{\hPG}{\mathsf{h}{:}\mathsf{P\!G}}

\newcommand{\iTG}{\mathsf{i}{:}\!\mathsf{T\!G}}
\newcommand{\eTG}{\mathsf{e}{:}\!\mathsf{T\!G}}
\newcommand{\hTG}{\mathsf{h}{:}\!\mathsf{T\!G}}
\newcommand{\phTG}{\mathsf{p}{:}\!\mathsf{T\!G}}
\newcommand{\cTG}{\mathsf{c}{:}\!\mathsf{T\!G}}

\newcommand{\ipTG}{\mathsf{i}{:}\mathsf{p\!T\!G}}
\newcommand{\epTG}{\mathsf{e}{:}\mathsf{p\!T\!G}}
\newcommand{\hpTG}{\mathsf{h}{:}\mathsf{p\!T\!G}}
\newcommand{\ppTG}{\mathsf{p}{:}\mathsf{p\!T\!G}}
\newcommand{\cpTG}{\mathsf{c}{:}\mathsf{p\!T\!G}}

\newcommand{\iqTG}{\mathsf{i}{:}\mathsf{q\!T\!G}}
\newcommand{\eqTG}{\mathsf{e}{:}\mathsf{q\!T\!G}}
\newcommand{\hqTG}{\mathsf{h}{:}\mathsf{q\!T\!G}}
\newcommand{\pqTG}{\mathsf{p}{:}\mathsf{q\!T\!G}}
\newcommand{\cqTG}{\mathsf{c}{:}\mathsf{q\!T\!G}}

\newcommand{\erTG}{\mathsf{e}{:}\mathsf{r\!T\!G}}

\newcommand{\esTG}{\mathsf{e}{:}\mathsf{s\!T\!G}}
\newcommand{\isTG}{\mathsf{i}{:}\mathsf{s\!T\!G}}
\newcommand{\hsTG}{\mathsf{h}{:}\mathsf{s\!T\!G}}
\newcommand{\psTG}{\mathsf{p}{:}\mathsf{s\!T\!G}}
\newcommand{\csTG}{\mathsf{c}{:}\mathsf{s\!T\!G}}

\newcommand{\iTS}{\mathsf{i}{:}\!\mathsf{T\!S}}
\newcommand{\eTS}{\mathsf{e}{:}\!\mathsf{T\!S}}
\newcommand{\hTS}{\mathsf{h}{:}\!\mathsf{T\!S}}
\newcommand{\phTS}{\mathsf{p}{:}\!\mathsf{T\!S}}
\newcommand{\cTS}{\mathsf{c}{:}\!\mathsf{T\!S}}

\newcommand{\ipTS}{\mathsf{i}{:}\mathsf{p\!T\!S}}
\newcommand{\epTS}{\mathsf{e}{:}\mathsf{p\!T\!S}}
\newcommand{\hpTS}{\mathsf{h}{:}\mathsf{p\!T\!S}}
\newcommand{\ppTS}{\mathsf{p}{:}\mathsf{p\!T\!S}}
\newcommand{\cpTS}{\mathsf{c}{:}\mathsf{p\!T\!S}}

\newcommand{\isTS}{\mathsf{i}{:}\mathsf{s\!T\!S}}
\newcommand{\esTS}{\mathsf{e}{:}\mathsf{s\!T\!S}}
\newcommand{\hsTS}{\mathsf{h}{:}\mathsf{s\!T\!S}}
\newcommand{\psTS}{\mathsf{p}{:}\mathsf{s\!T\!S}}
\newcommand{\csTS}{\mathsf{c}{:}\mathsf{s\!T\!S}}

\title{Categorically closed topological groups}
\author{Taras Banakh}
\address{T.~Banakh: Ivan Franko National University of Lviv
(Ukraine) and\newline\indent Jan Kochanowski University in Kielce
(Poland)}
\email{t.o.banakh@gmail.com}

\begin{document}
\begin{abstract} Let $\vec\C$ be a category whose objects are semigroups with topology and morphisms are closed semigroup relations, in particular, continuous homomorphisms.
An object $X$ of the category $\vec\C$ is called {\em $\vec\C$-closed} if for each morphism $\Phi\subset X\times Y$ in the category $\vec\C$ the image $\Phi(X)=\{y\in Y:\exists x\in X\;(x,y)\in\Phi\}$ is closed in $Y$. In the paper we survey existing and new results on topological groups, which are $\vec\C$-closed for various categories $\vec\C$ of topologized semigroups.
\end{abstract}
\maketitle

\section{Introduction and Survey of Main Results}

In this paper we recognize topological groups which are $\vec\C$-closed for some categories $\vec\C$ of Hausdorff topologized semigroups.

A {\em topologized semigroup} is a topological space $S$ endowed
with an associative binary operation $S\times S\to
S$, $(x,y)\mapsto xy$. If the binary operation is (separately) continuous, then $S$ is called a ({\em semi}){\em topological semigroup}. A topologized semigroup $S$ is called {\em powertopological} if it is semitopological and for every $n\in\IN$ the map $S\to S$, $x\mapsto x^n$, is continuous. A topologized semigroup $S$ is called {\em right-topological} if for every $a\in S$ the right shift $S\to S$, $x\mapsto xa$, is continuous.

{\em All topologized semigroups considered in this paper} (except for those in Proposition~\ref{p:sd} and Example~\ref{e:Iso}) {\em are assumed to be Hausdorff}.  

Topologized semigroups are objects of many categories which differ by morphisms. The most obvious category for morphisms has continuous homomorphisms between topologized semigroups. A bit wider category for morphisms has partial homomorphisms, i.e., homomorphisms defined on subsemigroups. The widest category for morphisms has semigroup relations. By a {\em semigroup relation} between semigroups $X,Y$ we understand a subsemigroup $R\subset X\times Y$ of the product semigroup $X\times Y$.

Now we recall some standard operations on (semigroup) relations. For two (semigroup) relations $\Phi\subset X\times Y$ and $\Psi\subset Y\times Z$ their composition is the (semigroup) relation $\Psi\circ\Phi\subset X\times Z$ defined by $\Psi\circ\Phi=\{(x,z)\in X\times Z:\exists y\in Y\;\;(x,y)\in\Phi,\;(y,z)\in\Psi\}$. For a (semigroup) relation $R\subset X\times Y$ its inverse $R^{-1}$ is the (semigroup) relation $R^{-1}=\{(y,x):(x,y)\in R\}\subset Y\times X$.

For a relation $R\subset X\times Y$ and subsets $A\subset X$ the set  $R(A)=\{y\in X:\exists a\in A\;(a,y)\in R\}$ is the {\em image} of $A$ under the relation $R$. If $R\subset X\times Y$ is a semigroup relation between semigroups $X,Y$, then for any subsemigroup $A\subset X$ its image $R(A)$ is a subsemigroup of $Y$.
For a relation $R\subset X\times Y$ the sets $R(X)$ and $R^{-1}(X)$ are called the {\em range} and {\em domain} of $R$, respectively.

Semigroup relations between semigroups can be equivalently viewed as multimorphisms. By a {\em multimorphism} between semigroups $X,Y$ we understand a multi-valued function $\Phi:X\multimap Y$ such that $\Phi(x)\cdot\Phi(y)\subset \Phi(xy)$ for any $x,y\in X$. Observe that a multi-valued function  $\Phi:X\multimap Y$ between semigroups is a multimorphism if and only if its graph $\Gamma=\{(x,y)\in X\times Y:y\in\Phi(x)\}$ is a subsemigroup in $X\times Y$. Conversely, each subsemigroup $\Gamma\subset X\times Y$ determines a multimorphism $\Phi:X\multimap Y$, $\Phi:x\mapsto \Phi(x):=\{(y\in Y:(x,y)\in \Gamma\}$.  In the sequel we shall identify multimorphisms with their graphs. 

A multimorphism $\Phi:X\multimap Y$ between semigroups $X,Y$ is called a {\em partial homomorphism} if for each $x\in X$ the set $\Phi(x)$ contains at most one point. Each partial homomorphism $\Phi:X\multimap Y$ can be identified with the unique function $\varphi:\dom(\varphi)\to Y$ such that $\Phi(x)=\{\varphi(x)\}$ for each $x\in\dom(\varphi):=\Phi^{-1}(Y)$. This function $\varphi$ is a homomorphism from the subsemigroup $\dom(\varphi)$ to the semigroup $Y$.

For a class $\C$ of Hausdorff topologized semigroups by ${\mathsf c}{:}\C$ we denote the category whose objects are topologized semigroups in the class $\C$ and morphisms are closed semigroup relations between the topologized semigroups in the class $\C$. The category
${\mathsf c}{:}\C$ contains the subcategories ${\mathsf e}{:}\C$, ${\mathsf i}{:}\C$, ${\mathsf h}{:}\C$, and $\pC$  whose objects are topologized semigroups in the class $\C$ and morphisms are isomorphic topological embeddings, injective continuous homomorphisms, continuous homomorphisms, and partial continuous homomorphisms with closed domain, respectively.

In this paper we consider some concrete instances of the following general notion.

\begin{definition}
Let $\vec \C$ be a category of topologized semigroups and their semigroup relations. A topologized
semigroup $X$ is called {\em $\vec\C$-closed} if for any morphism $\Phi\subset X\times Y$ of the category $\vec \C$ the range $\Phi(X)$ is closed in $Y$.
\end{definition}

In particular, for a class $\C$ of topologized semigroups, a topologized semigroup $X$ is called
\begin{itemize}
\item {\em ${\mathsf e}{:}\C$-closed} if for each isomorphic topological embedding $f:X\to Y\in\C$ the image $f(X)$ is closed in $Y$;
\item {\em $\mathsf{i}{:}\C$-closed} if for any injective
continuous homomorphism $f:X\to Y\in\C$ the image $f(X)$ is closed in $Y$;
\item {\em $\mathsf{h}{:}\C$-closed} if for any continuous
homomorphism $f:X\to Y\in\C$ the
image $f(X)$ is closed in $Y$;
\item {\em $\pC$-closed} if for any continuous
homomorphism $f:Z\to Y\in\C$ defined on a closed subsemigroup $Z\subset X$ the
image $f(Z)$ is closed in $Y$;
\item {\em $\cC$-closed} if for any topologized semigroup $Y\in\C$ and any closed subsemigroup $\Phi\subset X\times Y$ the range $\Phi(X):=\{y\in Y:\exists x\in X\;\;(x,y)\in\Phi\}$ of $\Phi$ is closed in $Y$.
\end{itemize}

It is clear that for any class $\C$ of Hausdorff topologized semigroups and a topologized semigroup $X$ we have the implications:
$$\mbox{$\cC$-closed}\;\Ra\;\mbox{$\pC$-closed}\;\Ra\;\mbox{${\mathsf h}{:}\C$-closed}\;\Ra\;\mbox{${\mathsf i}{:}\C$-closed}\;\Ra\;\mbox{${\mathsf e}{:}\C$-closed}.$$


In this paper we are interested in characterizing topological groups which are ${\mathsf e}{:}\C$-, ${\mathsf i}{:}\C$-, $\hC$-, $\pC$- or $\cC$-closed for the following classes of Hausdorff topologized semigroups:
\begin{itemize}
\item $\TS$ of all topological semigroups,
\item $\pTS$ of all powertopological semigroups,
\item $\sTS$ of all semitopological semigroups,
\item $\rTS$ of all right-topological semigroups,
\smallskip
\item $\TG$ of all topological groups,
\item $\pTG$ of all paratopological groups,
\item $\qTG$ of all quasitopological groups,
\item $\sTG$ of all semitopological groups,
\item $\rTG$ of all right-topological groups.
\end{itemize}

We recall that a {\em paratopological group} is a group $G$ endowed with a topology making it a topological semigroup. So, the inversion operation is not necessarily continuous. A {\em quasitopological group} if a topologized group $G$ such that for any $a,b\in G$ and $n\in\{1,-1\}$ the map $G\to G$, $x\mapsto ax^nb$, is continuous.

The inclusion relations between the classes of topologized semigroups are described in the following diagram (in which an arrow $\mathsf A\to\mathsf B$ between classes $\mathsf A,\mathsf B$ indicates that $\mathsf A\subset\mathsf B$).
$$
\xymatrix{
\pTG\ar[d]&\TG\ar[r]\ar[l]&\qTG\ar[r]&\sTG\ar[d]\ar[r]&\rTG\ar[d]\\
\TS\ar[rr]&&\pTS\ar[r]&\sTS\ar[r]&\rTS
}
$$
In this paper we shall survey exising and new results on the following general problem (consisting of $9\times 5=45$ subproblems).

\begin{problem}\label{prob:gen} Given a class $\C\in\{\TS,\pTS,\sTS,\rTS,\,\TG,\qTG,\pTG,\sTG,\rTG\}$ and a class of morphisms $\mathsf{f}\in\{\mathsf e,\mathsf i,\mathsf h, \mathsf p, \mathsf c\}$ detect topological groups which are $\mathsf{f}\,\!{:}\C$-closed.
\end{problem}

For the categories $\eTG$ and $\eqTG$ the answer to this problem is known and is a combined result of Raikov \cite{Raikov1946} who proved the equivalence $(1)\Leftrightarrow(2)$ and Bardyla, Gutik, Ravsky \cite{BGR} who proved the equivalence $(1)\Leftrightarrow(3)$.

\begin{theorem}[Raikov, Bardyla--Gutik--Ravsky]\label{t:Raikov} For a topological group $X$ the following conditions are equivalent:
\begin{enumerate}
\item $X$ is $\eTG$-closed;
\item $X$ is $\eqTG$-closed;
\item $X$ is complete.
\end{enumerate}
\end{theorem}

A topological group $X$ is {\em complete} if it is complete in its two-sided uniformity, i.e., the uniformity, generated by the entourages $\{(x,y)\in X\times X:y\in xU\cap Ux\}$ where $U$ runs over neighborhood of the unit in $X$.

On the other hand, Gutik \cite[2.5]{Gutik14} answered Problem~\ref{prob:gen} for the category $\esTS$:

\begin{theorem}[Gutik]\label{t:Gutik} A topological group is compact if and only if it is $\esTS$-closed.
\end{theorem}


Theorems~\ref{t:Raikov}, \ref{t:Gutik} and the trivial inclusions $\qTG\supset\TG\subset\pTG\subset\TS\subset\pTS\subset\sTS$ imply the following diagram of implications between various $\vec\C$-closedness properties of a topological group:
$$
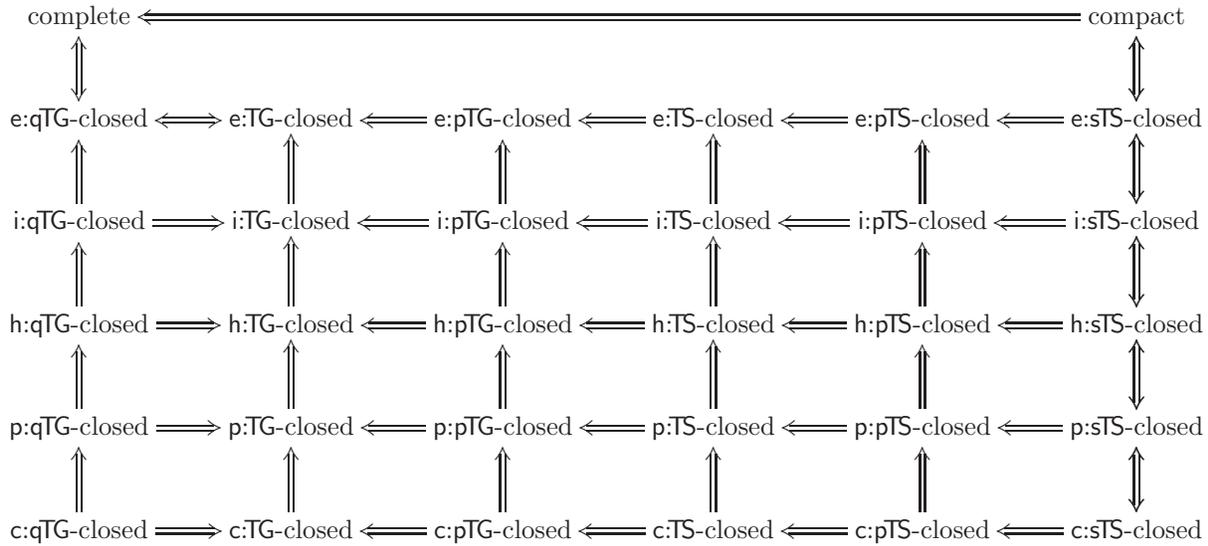
\xymatrix{
\mbox{complete}\ar@{<=>}[d]&&&&&\mbox{compact}\ar@{<=>}[d]\ar@{=>}[lllll]\\
\mbox{$\eqTG$-closed}\ar@{<=>}[r]&\mbox{$\eTG$-closed}&\mbox{$\epTG$-closed}\ar@{=>}[l]&\mbox{$\eTS$-closed}\ar@{=>}[l]
&\mbox{$\epTS$-closed}\ar@{=>}[l]&\mbox{$\esTS$-closed}\ar@{=>}[l]\ar@{<=>}[u]\\
\mbox{$\iqTG$-closed}\ar@{=>}[r]\ar@{=>}[u]&\mbox{$\iTG$-closed}\ar@{=>}[u]&\mbox{$\ipTG$-closed}\ar@{=>}[u]\ar@{=>}[l]
&\mbox{$\iTS$-closed}\ar@{=>}[u]\ar@{=>}[l]&\mbox{$\ipTS$-closed}\ar@{=>}[u]\ar@{=>}[l]&\mbox{$\isTS$-closed}\ar@{<=>}[u]
\ar@{=>}[l]\ar@{<=>}[u]\\
\mbox{$\hqTG$-closed}\ar@{=>}[u]\ar@{=>}[r]&\mbox{$\hTG$-closed}\ar@{=>}[u]&\mbox{$\hpTG$-closed}\ar@{=>}[u]\ar@{=>}[l]
&\mbox{$\hTS$-closed}\ar@{=>}[u]\ar@{=>}[l]&\mbox{$\hpTS$-closed}\ar@{=>}[u]\ar@{=>}[l]&\mbox{$\hsTS$-closed}\ar@{<=>}[u]
\ar@{=>}[l]\ar@{<=>}[u]\\
\mbox{$\pqTG$-closed}\ar@{=>}[u]\ar@{=>}[r]&\mbox{$\phTG$-closed}\ar@{=>}[u]&\mbox{$\ppTG$-closed}\ar@{=>}[u]\ar@{=>}[l]
&\mbox{$\phTS$-closed}\ar@{=>}[u]\ar@{=>}[l]&\mbox{$\ppTS$-closed}\ar@{=>}[u]\ar@{=>}[l]&\mbox{$\psTS$-closed}\ar@{<=>}[u]
\ar@{=>}[l]\ar@{<=>}[u]\\
\mbox{$\cqTG$-closed}\ar@{=>}[u]\ar@{=>}[r]&\mbox{$\cTG$-closed}\ar@{=>}[u]&\mbox{$\cpTG$-closed}\ar@{=>}[u]\ar@{=>}[l]
&\mbox{$\cTS$-closed}\ar@{=>}[u]\ar@{=>}[l]&\mbox{$\cpTS$-closed}\ar@{=>}[u]\ar@{=>}[l]&\mbox{$\csTS$-closed}\ar@{<=>}[u]
\ar@{=>}[l]\ar@{<=>}[u]
}
$$
This diagram shows that various $\vec \C$-closedness properties of topological groups fill and organize the ``space'' between compactness and completeness. 

In fact, under different names, $\vec\C$-closed topological (semi)groups have been already considered in mathematical literature.
As we have already mentioned, $\eTG$-closed topological groups appeared in Raikov's characterization \cite{Raikov1946} of complete topological groups. $\eTS$-Closed and $\hTS$-closed topological semigroups were introduced in 1969 by Stepp \cite{Stepp1969,Stepp75} who called them maximal and absolutely maximal semigroups, respectively. The study $\hTG$-closed, $\phTG$-closed and $\cTG$-closed topological groups (called $h$-complete, hereditarily $h$-complete and $c$-compact topological groups, respectiely) was initiated by Dikranjan and Tonolo \cite{DTon} and continued by Dikranjan, Uspenskij \cite{DU}, see the monograph of Luka\'sc  \cite{Luk} and survey \cite[\S4]{DSh} of Dikranjan and Shakhmatov.
The study of $\epTG$-closed paratopological groups was initiated by Banakh and Ravsky \cite{BanRav2001}, \cite{Ravsky2003},  who called them $H$-closed paratopological groups.  In \cite{Bardyla-Gutik-2012,
BGR,ChuchmanGutik2007,
GutikRepovs2008}
Hausdorff $\eTS$-closed (resp. $\hTS$-closed) topological semigroups are called (absolutely) $H$-closed. In \cite{Gutik14} Gutik studied and characterized $\esTS$-closed topological groups (calling them $H$-closed topological groups in the class of semitopological semigroups). The papers \cite{BB-sl1}, \cite{BB-sl2} are devoted to recognizing $\vec \C$-closed topological semilattices for various categories $\vec\C$ of topologized semigroups. In the paper \cite{Ban17} the author studied $\vec\C$-closedness properties in Abelian topological groups and proved the following characterization.

\begin{theorem}[Banakh]\label{t:banakh} An Abelian topological group $X$ is compact if and only if $X$ is $\iTG$-closed.
\end{theorem}

In Corollary~\ref{c:AesTG} we shall complement this theorem proving that an Abelian topological group is compact if and only if it is $\esTG$-closed.

The results of Banakh \cite{Ban17} and Ravsky~\cite{Ravsky2003} combined with Theorem~\ref{t:AbelC} (proved in this paper) imply the following characterization of Abelian topological groups which are $\epTG$-, $\eTS$- or $\epTS$-closed.

\begin{theorem}[Banakh, Ravsky]\label{t:BRintro} For an Abelian topological group $X$ the following conditions are equivalent:
\begin{enumerate}
\item $X$ is $\epTG$-closed;
\item $X$ is $\eTS$-closed;
\item $X$ is $\epTS$-closed;
\item $X$ is complete and has compact exponent;
\item $X$ is complete and for every  injective continuous homomorphism $f:X\to Y$ to a topological group $Y$ the group $\overline{f(X)}/f(X)$ is periodic.
\end{enumerate}
\end{theorem}

A group $X$ is called {\em periodic} if each element of $X$ has finite order. Theorem~\ref{t:BRintro}(4) involves the (important) notion of a topological groups of compact exponent, which is defined as follows.

\begin{definition} A topological group $X$ has ({\em pre}){\em compact exponent} if for some $n\in\IN$ the set $nX:=\{x^n:x\in X\}$ has compact closure in $X$ (resp. is totally bounded in $X$).
\end{definition}

Theorems~\ref{t:banakh}, \ref{t:BRintro} and Corollary~\ref{c:esTG} imply that for Abelian groups, the diagram describing the interplay between various $\vec \C$-closedness properties collapses to the following form (containing only three different types of closedness: compactness, completeness, and completeness combined with compact exponent):
{\small
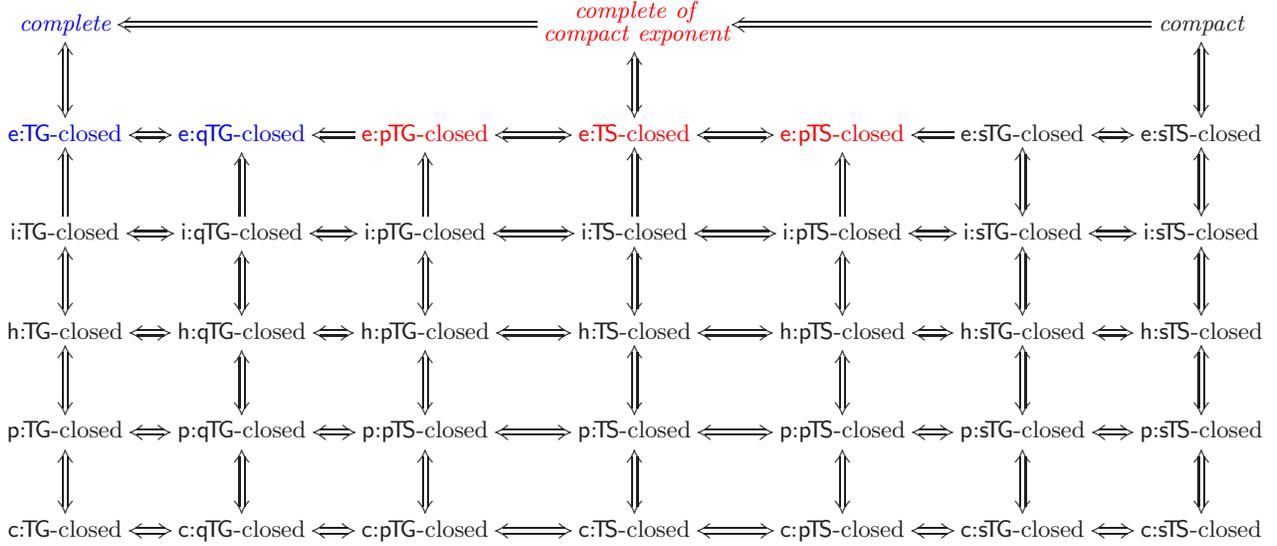
$$
\xymatrixcolsep{15pt}
\xymatrix{
\mbox{\color{blue}\it complete}\ar@{<=>}[d]&&&\mbox{\color{red}\small\it complete of}\atop\mbox{\color{red}\small\it \!compact exponent\!\!}\ar@{<=>}[d]\ar@{=>}[lll]&&&\mbox{\it compact}\ar@{<=>}[d]\ar@{=>}[lll]\\
\mbox{\color{blue}$\eTG$-closed}
&\mbox{\color{blue}$\eqTG$-closed}\ar@{<=>}[l]
&\mbox{\color{red}$\epTG$-closed}\ar@{=>}[l]
&\mbox{\color{red}$\eTS$-closed}\ar@{<=>}[l]
&\mbox{\color{red}$\epTS$-closed}\ar@{<=>}[l]
&\mbox{$\esTG$-closed}\ar@{=>}[l]
&\mbox{$\esTS$-closed}\ar@{<=>}[l]\\
\mbox{$\iTG$-closed}\ar@{=>}[u]
&\mbox{$\iqTG$-closed}\ar@{=>}[u]\ar@{<=>}[l]
&\mbox{$\ipTG$-closed}\ar@{=>}[u]\ar@{<=>}[l]
&\mbox{$\iTS$-closed}\ar@{=>}[u]\ar@{<=>}[l]
&\mbox{$\ipTS$-closed}\ar@{=>}[u]\ar@{<=>}[l]
&\mbox{$\isTG$-closed}\ar@{<=>}[l]\ar@{<=>}[u]
&\mbox{$\isTS$-closed}\ar@{<=>}[l]\ar@{<=>}[u]\\
\mbox{$\hTG$-closed}\ar@{<=>}[u]
&\mbox{$\hqTG$-closed}\ar@{<=>}[u]\ar@{<=>}[l]
&\mbox{$\hpTG$-closed}\ar@{<=>}[u]\ar@{<=>}[l]
&\mbox{$\hTS$-closed}\ar@{<=>}[u]\ar@{<=>}[l]
&\mbox{$\hpTS$-closed}\ar@{<=>}[u]\ar@{<=>}[l]
&\mbox{$\hsTG$-closed}\ar@{<=>}[u]\ar@{<=>}[l]
&\mbox{$\hsTS$-closed}\ar@{<=>}[u]\ar@{<=>}[l]\\
\mbox{$\phTG$-closed}\ar@{<=>}[u]
&\mbox{$\pqTG$-closed}\ar@{<=>}[u]\ar@{<=>}[l]
&\mbox{$\ppTS$-closed}\ar@{<=>}[u]\ar@{<=>}[l]
&\mbox{$\phTS$-closed}\ar@{<=>}[u]\ar@{<=>}[l]
&\mbox{$\ppTS$-closed}\ar@{<=>}[u]\ar@{<=>}[l]
&\mbox{$\psTG$-closed}\ar@{<=>}[u]\ar@{<=>}[l]
&\mbox{$\psTS$-closed}\ar@{<=>}[u]\ar@{<=>}[l]\\
\mbox{$\cTG$-closed}\ar@{<=>}[u]
&\mbox{$\cqTG$-closed}\ar@{<=>}[u]\ar@{<=>}[l]
&\mbox{$\cpTG$-closed}\ar@{<=>}[u]\ar@{<=>}[l]
&\mbox{$\cTS$-closed}\ar@{<=>}[u]\ar@{<=>}[l]
&\mbox{$\cpTS$-closed}\ar@{<=>}[u]\ar@{<=>}[l]
&\mbox{$\csTG$-closed}\ar@{<=>}[u]\ar@{<=>}[l]
&\mbox{$\csTS$-closed}\ar@{<=>}[u]\ar@{<=>}[l]
}
$$
}

So, the problem remains to investigate the $\vec\C$-closedness properties for non-commutative topological groups. Now we survey the principal results (known and new) addressing this complex and difficult problem. We start with the following characterization of $\eTS$-closed topological groups, proved in Section~\ref{s:eC}.

\begin{theorem}\label{t1.1} A topological group $X$ is $\eTS$-closed if and only if $X$ is Weil-complete and for every continuous homomorphism $f:X\to Y$ into a Hausdorff topological semigroup $Y$ the complement $\overline{f(X)}\setminus f(X)$ is not an ideal in the semigroup $\overline{f(X)}$.
\end{theorem}


Using Theorems~\ref{t:banakh}, \ref{t:BRintro} and \ref{t1.1} we shall prove that various $\vec\C$-closedness properties have strong implications on the structure of subgroups related to commutativity, such as the subgroups of the topological derived series or the central series of a given topological group.

 We recall that for a group $G$ its {\em commutator} $[G,G]$ is the subgroup generated by the set $\{xyx^{-1}y^{-1}:x,y\in G\}$.
The {\em topological derived series} $$G=G^{[0]}\supset G^{[1]}\supset G^{[2]}\supset\cdots$$of a topological group $G$ consists of the subgroups defined by the recursive formula $G^{[n+1]}:=\overline{[G^{[n]},G^{[n]}]}$ for $n\in\w$.
A topological group $G$ is called {\em solvable} if $G^{[n]}=\{e\}$ for some $n\in\IN$. The quotient group $X/X^{[1]}$ is called the {\em Abelization} of a topological group $X$.

The {\em central series}
$$\{e\}=Z_0(G)\subset Z_1(G)\subset\dots$$
of a (topological) group $G$ consists of (closed) normal subgroups defined by the recursive formula $$Z_{n+1}(G):=\{z\in X:\forall x\in G \;\;zxz^{-1}x^{-1}\in Z_n(G)\}\mbox{ \ for \ }n\in\w.$$A group $G$ is called {\em nilpotent} if $G=Z_n(G)$ for some $n\in\w$.
The subgroup $Z_1(G)$ is called the {\em center} of the group $G$ and is denoted by $Z(G)$.

The following theorem unifies Propositions~\ref{p:eTS-center}, and Corollaries~\ref{c:esTG}, \ref{c:hTS-center}.

\begin{theorem} Let $X$ be a topological group.
\begin{enumerate}
\item If $X$ is $\epTG$-closed, then the center $Z(X)$ has compact exponent.
\item If $X$ is $\eTS$-closed, then for any closed normal subgroup $N\subset X$ the center $Z(X/N)$ of the quotient topological group $X/N$ has precompact exponent.
\item If $X$ is $\iTG$-closed or $\esTS$-closed, then the center $Z(X)$ of $X$ is compact.
\item If $X$ is $\hTG$-closed, then  for any closed normal subgroup $N\subset X$ the center $Z(X/N)$ is compact; in particular, the Abelization $X/{X^{[1]}}$ of $X$ is compact.
\end{enumerate}
\end{theorem}

Applying the statements (2) and (4) inductively, we obtain the following corollary describing the compactness properties of some characteristic subgroups of a $\vec\C$-closed topological group (see
 Corollary~\ref{c:eTS=>Zn}, Proposition~\ref{p:hTG=>Z_n}, and Theorem~\ref{t:DU-hypocommutator}).

\begin{corollary}\label{c:Cclosed-Zn} Let $X$ be a topological group.
\begin{enumerate}
\item If $X$ is $\eTS$-closed, then for every $n\in\w$ the subgroup $Z_n(X)$ has compact exponent.
\item If $X$ is $\hTS$-closed, then for every $n\in\w$ the subgroup $Z_n(X)$ is compact.
\item If $X$ is $\phTG$-closed, then for every $n\in\w$ the quotient topological group $X/X^{[n]}$ is compact.
\end{enumerate}
\end{corollary}

The three items of Corollary~\ref{c:Cclosed-Zn} imply the following three characterizations. The first of them characterizes nilpotent complete group of compact exponent and is proved in Theorem~\ref{t:AbelC}.

\begin{theorem}\label{t:eTS-nilpotent} For a nilpotent topological group $X$ the following conditions are equivalent:
\begin{enumerate}
\item $X$ is complete and has compact exponent;
\item $X$ is $\eTS$-closed;
\item $X$ is $\epTS$-closed.
\end{enumerate}
\end{theorem}

In Example~\ref{e:Iso} we shall observe that the discrete topological group $\Iso(\IZ)$ of isometries of $\IZ$ is $\eTS$-closed but does not have compact exponent. This shows that Theorem~\ref{t:eTS-nilpotent} does not generalize to solvable groups.

\begin{theorem}[Dikranjan, Uspenskij]\label{t:DU-nilpotent} For a nilpotent topological group $X$ the following conditions are equivalent:
\begin{enumerate}
\item $X$ is compact;
\item $X$ is $\hTG$-closed.
\end{enumerate}
\end{theorem}
For Abelian topological groups Theorem~\ref{t:DU-nilpotent} was independently proved by Zelenyuk and Protasov \cite[4.3]{ZP}.

A topological group $X$ is called {\em hypoabelian} if for each non-trivial closed subgroup $X$ the commutator $[X,X]$ is not dense in $X$. It is easy to see that each solvable topological group is hypoabelian.

\begin{theorem}[Dikranjan, Uspenskij]\label{t:DU-solvable} For a solvable (more generally, hypoabelian) topological group $X$ the following conditions are equivalent:
\begin{enumerate}
\item $X$ is compact;
\item $X$ is $\phTG$-closed;
\item any closed subgroup of $X$ is $\hTG$-closed.
\end{enumerate}
\end{theorem}

The last two theorems were proved by Dikranjan and Uspenskij in \cite[3.9 and 3.10]{DU} (in terms of the $h$-completeness, which is an alternative name for the $\hTG$-closedness).

The Weyl-Heisenberg group $H(w_0)$ (which is a non-compact $\iTG$-closed nilpotent Lie group) shows that $\hTG$-closedness in Theorem~\ref{t:DU-nilpotent} cannot be weakened to the $\iTG$-closedness (see Example~\ref{ex:HW} for more details).

On the other hand, the solvable Lie group $\Iso_+(\IC)$ of orientation preserving isometries of the complex plane is $\hTS$-closed and not compact, which shows that the $\phTG$-closedness in Theorem~\ref{t:DU-solvable}(2) cannot be replaced by the $\hTG$-closedness of $X$. This example (analyzed in details in Section~\ref{s:Iso}) answers Question 3.13 in \cite{DU} and Question 36 in \cite{DSh}.
\smallskip

Nonetheless, the $\phTG$-closedness of the solvable group $X$ in Theorem~\ref{t:DU-solvable}(2) can be replaced by the $\hTG$-closedness of $X$ under the condition that the group $X$ is balanced and MAP-solvable.

A topological group $X$ is called {\em balanced} if for any neighborhood $U\subset X$ of the unit there exists a neighborhood $V\subset X$ of the unit such that $xV\subset Ux$ for all $x\in X$.

A topological group $X$ is called {\em maximally almost periodic} (briefly {\em MAP}) if it admits a continuous injective homomorphism $h:X\to K$ into a compact topological group $K$. By Theorem~\ref{t:MAP->i=h}, for any productive class $\C\supset\TG$ of topologized semigroups, {\em the $\iC$-closedness and $\hC$-closedness are equivalent for MAP topological groups.}

A topological group $X$ is defined to be {\em MAP-solvable} if there exists an icreasing sequence $\{e\}=X_0\subset X_1\subset\dots \subset X_m=X$ of closed normal subgroups in $X$ such that for every $n<m$ the quotient group $X_{n+1}/X_n$ is Abelian and MAP. Since locally compact Abelian groups are MAP, each solvable locally compact topological group is MAP-solvable.

The following theorem (proven in Section~\ref{s:MAP-solv}) nicely complements Theorem~\ref{t:DU-solvable} of Dikranjan and Uspenskij. Example~\ref{ex:solv} of non-compact solvable $\hTS$-closed Lie group $\Iso_+(\IC)$ shows that the ``balanced'' requirement cannot be removed from the conditions (2), (3).

\begin{theorem}\label{t:hypoabelian} For a solvable topological group $X$ the following conditions are equivalent:
\begin{enumerate}
\item $X$ is compact;
\item $X$ is balanced, locally compact, and $\hTG$-closed;
\item $X$ is balanced, MAP-solvable and $\hTG$-closed.
\end{enumerate}
\end{theorem}

It is interesting that  the proof of this theorem exploits a good piece of Descriptive Set Theory (that dealing with $K$-analytic spaces). Also methods of Descriptive Set Theory are used for establishing the interplay between $\iTG$-closed and minimal topological groups.

We recall that a topological group $X$ is {\em minimal} if each continuous bijective homomorphism $h:X\to Y$ to a topological group $Y$ is open (equivalently, is a topological isomorphism). By the fundamental theorem of Prodanov and Stoyanov \cite{Prod-Stoj}, each minimal topological Abelian group is precompact, i.e., has compact Raikov completion. Groups that are minimal in the discrete topology are called {\em non-topologizable}.
For more information on minimal topological groups we refer the reader to the monographs \cite{DPS}, \cite{Luk} and the surveys \cite{D98} and  \cite{DM}.

The definition of minimality implies that {\em  a minimal topological group is $\iTG$-closed if and only if it is  $\eTG$-closed if and only if it is complete}. In particular, each minimal complete topological group is $\iTG$-closed. By Theorem~\ref{t:banakh}, the converse implication holds for Abelian topological groups. It also holds for $\w$-narrow topological groups of countable pseudocharacter.

\begin{theorem}\label{t:analytic}  An $\w$-narrow topological group $X$ of countable pseudocharacter is $\iTG$-closed if and only if $X$ is complete and minimal.
\end{theorem}

A subset $B\subset X$ of a topological group $X$ is called {\em $\w$-narrow} if for any neighborhood $U\subset X$ of the unit there exists a countable set $C\subset X$ such that $B\subset CU\cap UC$. $\w$-Narrow topological groups were introduced by Guran \cite{Guran} (as $\aleph_0$-bounded groups) and play important role in Theory of Topological Groups \cite{AT}.
Theorem~\ref{t:analytic} will be proved in Section~\ref{s:iC} (see Theorem~\ref{t:pseudo}). This theorem suggests the following open problem.

\begin{problem}\label{prob:i-min} Is each $\iTG$-closed topological group minimal?
\end{problem}

Observe that a complete MAP topological group is minimal if and only if it is compact. So, for MAP topological groups Problem~\ref{prob:i-min} is equivalent to another intriguing open problem.

\begin{problem}\label{prob:i-MAP} Is each $\iTG$-closed MAP topological group compact?
\end{problem}

For $\w$-narrow topological groups an affirmative answer to this problem follows from Theorem~\ref{t:MAP->i=h} and the characterization of $\hTG$-closedness in term of total completeness and total minimality, see Theorem~\ref{t:hTG=tc+tm}.

Following \cite{Luk}, we define a topological group $G$ to be {\em totally complete} (resp. {\em totally minimal\/}) if for any closed normal subgroup $H\subset G$ the quotient topological group $G/H$ is complete (resp. minimal). Totally minimal topological groups were introduced by Dikranjan and Prodanov in \cite{DProd}. By \cite[3.45]{Luk}, each totally complete totally minimal topological group is absolutely $\TG$-closed.

\begin{theorem}\label{t:hTG=tc+tm} An $\w$-narrow topological group is $\hTG$-closed if and only if it is totally complete and totally minimal.
\end{theorem}

Theorem~\ref{t:hTG=tc+tm} will be proved in Section~\ref{s:hC} (see Theorem~\ref{t:absw}). This theorem complements a  characterization of $\hTG$-closed topological groups in terms of special filters, due to Dikranjan and Uspenskij \cite{DU} (see also \cite[4.24]{Luk}). Using their characterization of $\hTG$-closedness, Dikranjan and Uspenskij \cite[2.16]{DU} proved another characterization.

\begin{theorem}[Dikranjan, Uspenskij]\label{t:cTG} A balanced topological group is $\hTG$-closed if and only if it is $\cTG$-closed.
\end{theorem}

The compactness of $\w$-narrow $\iTG$-closed MAP topological groups can be also derived from the compactness of the {\em $\w$-conjucenter} $Z^\w(X)$ defined for any topological group $X$ as the set of all points $z\in X$ whose conjugacy class $C_X(z):=\{xzx^{-1}:x\in X\}$ is $\w$-narrow in $X$.

A topological group $X$ is defined to be {\em $\w$-balanced} if for any neighborhood $U\subset X$ of the unit there exists a countable family $\V$ of neighborhoods of the unit such that for any $x\in X$ there exists $V\in\V$ such that $xVx^{-1}\subset U$. It is known (and easy to see) that each $\w$-narrow topological group is $\w$-balanced. By Katz Theorem~\cite[3.4.22]{AT}, a topological group is $\w$-balanced if and only if it embeds into a Tychonoff product of first-countable topological groups.
The following theorem can be considered as a step towards the solution of Problem~\ref{prob:i-MAP}.

\begin{theorem}\label{t:conjucenter} If an $\w$-balanced MAP topological group $X$ is $\iTG$-closed, then its $\w$-conjucenter $Z^\w(X)$ is compact.
\end{theorem}

A topological group $G$ is called {\em hypercentral} if for each closed normal subgroup $H\subsetneq G$, the quotient group $G/H$ has non-trivial center $Z(G/H)$. It is easy to see that each nilpotent topological group is hypercentral. Theorem~\ref{t:conjucenter} implies the following characterization (see Corollary~\ref{c:hypercentral}).

\begin{corollary}\label{c:hyperabelian} A hypercentral topological group $X$ is compact if and only if $\w$-balanced, MAP, and $\iTG$-closed.
\end{corollary}


\begin{remark} Known examples of non-topologizable groups (due to Klyachko, Olshanskii, and Osin \cite{KOO}) show that the compactness does not follow from the $\pTS$- or $\cTG$-closedness  even for 2-generated discrete topological groups (see Example~\ref{ex:non-top}).
\end{remark}

The following diagram describes the implications between various completeness and closedness properties of a topological group. By simple arrows we indicate the implications that hold under some additional assumptions (written in italic near the arrow).
$$
\xymatrixcolsep{15pt}
\xymatrixrowsep{30pt}
\xymatrix{
\mbox{complete}\ar@{<=>}[rr]&&\mbox{$\eTG$-closed}\ar@/^10pt/^{\mbox{\it\tiny minimal}}[d]&\mbox{$\epTG$-closed}\ar@{=>}[l]\ar@/_12pt/_{\mbox{\it\tiny Abelian}}[r]&\mbox{$\eTS$-closed}\ar_{\mbox{\tiny\it nilpotent}}[rd]\ar@{=>}[l]
&\mbox{$\epTS$-closed}\ar@{=>}[l]\\
\mbox{complete}\atop\mbox{minimal}
\ar@/_10pt/_{\mbox{\it\tiny MAP}}[d]
\ar@{=>}[rr]\ar@{=>}[u]
&&\mbox{$\iTG$-closed}\ar@/^16pt/^{\mbox{\it\tiny of compact}\atop\mbox{\tiny\it exponent}}[rr]\ar@/^10pt/^{\mbox{\it\tiny Abelian}}[d]\ar@/_10pt/_{\mbox{\it\tiny MAP}}[d]
\ar@{=>}[u]\ar@/_12pt/_{\mbox{\it\tiny $\w$-narrow of countable}\atop\mbox{\tiny\it pseudocharacter}}[ll]
&&\mbox{$\iTS$-closed}\ar@{=>}[ll]\ar@{=>}[u]\ar@/^10pt/^{\mbox{\it\tiny MAP}}[d]
&\mbox{\small complete of}\atop\mbox{\small compact exponent}\ar@{=>}[u]\\
\mbox{\small totally complete}\atop\mbox{\small totally minimal}\ar@{=>}[rr]\ar@/_10pt/_{\mbox{\it\tiny MAP}}[dd]
\ar@{=>}[u]&&\mbox{$\hTG$-closed}\ar@/_16pt/_{\mbox{\it\tiny of compact}\atop\mbox{\tiny\it exponent}}[rr]\ar@{=>}[u]
\ar@/_12pt/_{\mbox{\it\tiny $\w$-narrow}}[ll]
\ar@/_10pt/_{\mbox{\it\tiny MAP-solvable\!\!\!\!\!\!\!\!}\atop\mbox{\tiny\it  balanced\!\!\!\!\!\!\!\!\!}}^{\mbox{\it\tiny \!\!nilpotent}}[ddll]
&&\mbox{$\hTS$-closed}\ar@{=>}[ll]\ar@{=>}[u]
&\mbox{compact}\ar@{=>}[u]\\
&&\mbox{$\phTG$-closed}\ar@/^16pt/[rr]
\ar_{\mbox{\it\tiny MAP}}^{\mbox{\it\tiny \phantom{m}hypoabelian}}[lld]
\ar@{=>}[u]\ar@/^10pt/^{\mbox{\it\tiny balanced}}[d]&&\mbox{$\phTS$-closed}\ar@{=>}[ll]\ar@{=>}[u]
&\mbox{$\esTS$-closed}\ar@{<=>}[u]
\\
\mbox{compact}\ar@{=>}[rr]\ar@{=>}[uu]&&\mbox{$\cTG$-closed}\ar@{=>}[u]&&\mbox{$\cTS$-closed}\ar@{=>}[ll]\ar@{=>}[u]
&\mbox{$\erTG$-closed}\ar@{=>}[l]\ar@{<=>}[u]
}
$$
The curved horizontal implications, holding under the assumption of compact exponent, are proved in Theorems~\ref{t:i-exp} and \ref{t:exp=>hTS=hTG}.

\section{Completeness of topological groups versus $\C$-closedness}

To discuss the completeness properties of topological groups, we need to recall some known information related to uniformities on topological groups (see \cite{RD} and \cite[\S1.8]{AT} for more details). We refer the reader to \cite[Ch.8]{Eng} for basic information on uniform spaces. Here we recall that a uniform space $(X,\U)$ is {\em complete} if each Cauchy filter $\F$ on $X$ converges to some point $x\in X$. A {\em filter} on a set $X$ is a non-empty family of non-empty subsets of $X$, which is closed under finite intersections and taking supersets. A subfamily $\mathcal B\subset\F$ is called a {\em base} of a filter $\F$ if each set $F\in\F$ contains some set $B\in\mathcal B$.

A filter $\F$ on a uniform space $(X,\U)$ is {\em Cauchy} if for each entourage $U\in\U$ there is a set $F\in\F$ such that $F\times F\subset U$. A filter on a topological space $X$ {\em converges} to a point $x\in X$ if each neighborhood of $x$ in $X$ belongs to the filter. A uniform space $(X,\U)$ is compact if and only if the space is complete and {\em totally bounded} in the sense that for every entourage $U\in\U$ there exists a finite subset $F\subset X$ such that $X=\bigcup_{x\in F}B(x,U)$ where $B(x,U):=\{y\in X:(x,y)\in U\}$.

Each topological group $(X,\tau)$ with unit $e$ carries four natural uniformities:
\begin{itemize}
\item the {\em left uniformity} $\U_l$ generated by the base $\mathcal B_l=\big\{\{(x,y)\in X\times X:y\in xU\}:e\in U\in\tau\big\}$;
\item the {\em right uniformity} $\U_r$ generated by the base $\mathcal B_r=\big\{\{(x,y)\in X\times X:y\in Ux\}:e\in U\in\tau\big\}$;
\item the {\em two-sided uniformity} $\U_{\vee}$ generated by the base
$\mathcal B_{\vee}=\big\{\{(x,y)\in X\times X:y\in Ux\cap xU\}:e\in U\in\tau\big\}$;
\item the {\em Roelcke uniformity} $\U_{\wedge}$ generated by  the base
$\mathcal B_{\wedge}=\big\{\{(x,y)\in X\times X:y\in UxU\}:e\in U\in\tau\big\}$.
\end{itemize}
It is well-known (and easy to see) that a topological group $X$ is complete in its left uniformity if and only if it is complete in its right uniformity. Such topological groups are called {\em Weil-complete}. A topological group is {\em complete} if it is complete in its two-sided uniformity. Since each Cauchy filter in the two-sided uniformity is Cauchy in the left and right uniformities, each Weil-complete topological group is complete. For an Abelian (more generally, balanced) topological group $X$ the uniformities $\U_l$, $\U_r$, $\U_t$ coincide, which implies that $X$ is Weil-complete if and only if it is complete. 

An example of a complete topological group, which is not Weil-complete is the Polish group $\mathrm{Sym}(\w)$ of all bijections of the discrete countable space $\w$ (endowed with the topology of pointwise convergence, inherited from the Tychonoff product $\w^\w$).

The completion of a topological group $X$ by its two-sided uniformity is called the {\em  Raikov-completion} of $X$. It is well-known that the Raikov-completion of a topological group has a natural structure of a topological group, which contains $X$ as a dense subgroup.   On the other hand, the completion of a topological group $X$ by its left (or right)  uniformity carries a natural structure of a topological semigroup, called the (left or right) {\em Weil-completion} of the topological group, see \cite[8.45]{Strop}. For example, the left Weil-completion of the Polish group $\mathrm{Sym}(\w)$ is the semigroup of all injective functions from $\w$ to $\w$.

So, if a topological group $X$ is not complete, then $X$ admits a non-closed embedding into its Raikov-completion, which implies that it is not $\eC$-closed for any class $\C$ of topologized semigroups, containing all complete topological groups.
If $X$ is not Weil-complete, then $X$ admits a non-closed embedding into its (left or right) Weil-completion, which implies that it is not $\eC$-closed for any class $\C$ of topologized semigroups, containing all Tychonoff topological semigroups. Let us write these facts for future references.

\begin{theorem} \label{t:Weil}
 Assume that a class $\C$ of topologized semigroups contains all Raikov-completions (and Weil-completions) of topological groups. Each $\eC$-closed topological group is (Weil-)complete. In particular, each $\eTG$-closed topological group is complete and each $\eTS$-closed topological group is Weil-complete.
\end{theorem}

We recall that a non-empty subset $I$ of a semigroup $S$ is called an {\em ideal} in $S$ if  $IS\cup SI\subset I$.

\begin{theorem}\label{t:2.5n} Assume that a topological group $X$ admits a non-closed topological isomorphic embedding $f:X\to Y$ into a Hausdorff semitopological semigroup $Y$.
\begin{enumerate}
\item If $X$ is Weil-complete, then $\overline{f(X)}\setminus f(X)$ is an ideal of the semigroup $\overline{f(X)}$.
\item If $X$ is complete, then $\{y^n:y\in \overline{f(X)}\setminus f(X),\;n\in\IN\}\subset \overline{f(X)}\setminus f(X)$.
\end{enumerate}
\end{theorem}

\begin{proof} To simplify notation, it will be convenient to identify $X$ with its image $f(X)$ in $Y$. Replacing $Y$ by the closure $\bar X$ of $X$, we can assume that the group $X$ is dense in the semigroup $Y$.
\smallskip

1. First, we assume that $X$ is Weil-complete. Given any $x\in X$ and $y\in Y\setminus X$, we should prove that $xy,yx\in Y\setminus X$.

To derive a contradiction, assume that $xy\in X$. On the topological group $X$, consider the filter $\F$ generated by the base consisting of the intersections $X\cap O_{y}$ of $X$ with neighborhoods $O_y$ of $y$ in $Y$. The Hausdorff property of $Y$ ensures that this filter does not converge in the Weil-complete group $X$ and thus is not Cauchy in the left uniformity of $X$. This yields an open neighborhood $V_e$ of the unit of the group $X$ such that $F\not\subset zV_e$ for any set $F\in\F$ and any $z\in X$. Since $X$ carries the subspace topology, the space $Y$ contains an open set $U_{xy}\subset Y$ such that $U_{xy}\cap X=xyV_e$.

The separate continuity of the binary operation on $Y$ yields an open neighborhood $U_{x}\subset Y$ of the point $x$ in $Y$ such that $U_xy\subset U_{xy}$. Choose any point $z\in U_x$ and find a neighborhood $U_y$ of the point $y$ in $Y$ such that $zU_y\subset U_{xy}$. Now consider the set $F=X\cap U_y\in \F$ and observe that $zF\subset X\cap U_{xy}=xyV_e$ and hence $F\subset z^{-1}xyV_e$, which contradicts the choice of $V_e$. This contradiction shows that $xy\in Y\setminus X$.

By analogy we can prove that $yx\in Y\setminus X$.
\smallskip

2. Next, assume that $X$ is complete. In this case we should prove that $y^n\notin X$ for any $y\in Y\setminus X$ and $n\in\IN$. To derive a contradiction, assume that $y^n\in X$ for some $y\in Y\setminus X$ and $n\in\IN$. On the group $X$ consider the filter $\F$ generated by the base $X\cap O_y$ where $O_y$ runs over neighborhoods of $y$ in $Y$. The filter $\F$ converges to the point $y\notin X$ and hence is divergent in $X$ (by the Hausdorff property of $Y$). Since $X$ is complete, the divergent filter $\F$ is not Cauchy in its two-sided uniformity. This allows us to find an open neighborhood $V_e\subset X$ of the unit such that $F\not\subset xV_e\cap V_ez$ for any points $x,z\in X$. Choose an open set $W\subset Y$ such that $W\cap X=y^nV_e\cap V_ey^n$.

By finite induction, we shall construct a sequence $(x_i)_{i=0}^{n-1}$ of points of the group $X$ such that $x_iy^{n-i}\in W$ for all $i\in\{1,\dots,n-1\}$. To start the inductive construction, let $x_0=e$ be the unit of the group $X$. Assume that for some positive $i\le n-1$ the point $x_{i-1}\in X$ with $x_{i-1}y^{n+1-i}\in W$ has been constructed. By the separate continuity of the semigroup operation in $Y$, the point $y$ has a neighborhood $V_y\subset Y$ such that $x_{i-1}V_yy^{n-i}\in W$. Choose any point $v\in X\cap V_y$, put $x_i=x_{i-1}v\in X\cap V_y$ and observe that $x_iy^{n-i}\in W$, which completes the inductive step.

After completing the inductive construction, we obtain a point $x=x_{n-1}\in X$ such that $xy\in W$. By analogy we can construct a point $z\in X$ such that $yz\in W$. The separate continuity of the binary operation in $Y$ yields a neighborhood $V_y\subset Y$ of $y$ such that $(xV_y)\cup(V_yz)\subset W$. Then the set $F=X\cap V_y\in\F$ has the property: $(xF)\cup(Fz)\subset X\cap W=y^nV_e\cap V_ey^n$ which implies that $F\subset (x^{-1}y^nV_e)\cap (V_ey^nz^{-1})$. But this contradicts the choice of the neighborhood $V_e$.
\end{proof}

Now we describe a construction of the ideal union of topologized semigroups, which allows us to construct non-closed embeddings of topologized semigroups.

Let $h:X\to Y$ be a continuous homomorphism between topologized semigroups $X,Y$ such that $Y\setminus h(X)$ is an ideal in $Y$ and $X\cap (Y\setminus h(X))=\emptyset$. Consider the set $U_h(X,Y):=X\cup (Y\setminus h(X))$ endowed with the semigroup operation defined by
$$xy=\begin{cases}
x*y&\mbox{if $x,y\in X$};\\
h(x)\cdot y&\mbox{if $x\in X$ and $y\in Y\setminus h(X)$};\\
x\cdot h(y)&\mbox{if $x\in Y\setminus h(X)$ and $y\in X$};\\
x\cdot y&\mbox{if $x,y\in Y\setminus h(X)$}.\\
\end{cases}
$$ Here by $*$ and $\cdot$ we denote the binary operations of the semigroups $X$ and $Y$, respectively.
The set $U_h(X,Y)$ is endowed with the topology consisting of the sets $W\subset U_h(X,Y)$ such that
\begin{itemize}
\item for any $x\in W\cap X$, some neighborhood $U_x\subset X$ of $x$ is contained in $W$;
\item for any $y\in W\cap (Y\setminus h(X))$ there exists an open neighborhood $U\subset Y$ of $y$ such that $h^{-1}(U)\cup(U\setminus h(X))\subset W$.
\end{itemize}
This topology turns $U_h(X,Y)$ into a topologized semigroup, which be called the {\em ideal union of the semigroups $X$ and $Y$ along the homomorphism $h$}.

The following lemma can be derived from the definition of the ideal union.

\begin{theorem}\label{t:2.2} Let $h:X\to Y$ be a continuous homomorphism between topologized semigroups such that $Y\setminus h(X)$ is an ideal in $Y$. The topologized semigroup $U_h(X,Y)$ has the following properties:
\begin{enumerate}
\item $X$ is an open subsemigroup of $U_h(X,Y)$;
\item $X$ is closed in $U_h(X,Y)$ if and only if $h(X)$ is closed in $Y$;
\item If $X$ and $Y$ are (semi)topological semigroups, then so is the topologized semigroup $U_h(X,Y)$;
\item If the spaces $X,Y$ are Hausdorff (or regular or Tychonoff), then so is the space $U_h(X,Y)$.
\end{enumerate}
\end{theorem}

We shall say that a class $\C$ of topologized semigroups is stable under taking
\begin{itemize}
\item {\em topological isomorphisms} if for any topological isomorphism $h:X\to Y$ between topologized semigroups $X,Y$ the inclusion $X\in\C$ implies $Y\in\C$;
\item {\em closures} if for any topologized semigroup $Y\in\C$ and a subgroup $X\subset Y$ the closure $\bar X$ of $X$ in $Y$ is a topologized semigroup that belongs to the class $\C$;
\item {\em ideal unions} if for any continuous homomorphism $h:X\to Y$ between semigroups $X,Y\in\C$ with  $Y\setminus h(X)$ being an ideal in $Y$, the topologized semigroup $U_h(X,Y)$ belongs to the class $\C$.
\end{itemize}

Theorems~\ref{t:2.5n} and \ref{t:2.2} imply the following characterization.

\begin{theorem}\label{t:2.3} Assume that a class $\C$ of Hausdorff semitopological semigroups is stable under topological isomorphisms, closures and ideal unions. A Weil-complete topological group $X\in\C$ is $\eC$-closed if and only if for every continuous homomorphism $f:X\to Y$ into a topologized semigroup $Y\in\C$ the set $\overline{f(X)}\setminus f(X)$ is not an ideal in $\overline{f(X)}$.
\end{theorem}

\begin{proof} To prove the ``only if'' part, assume that there exits a continuous homomorphism $f:X\to Y$ into a topologized semigroup $Y\in\C$ such that $f(X)$ is dense in $Y$ and $\overline{f(X)}\setminus f(X)$ is an ideal of $\overline{f(X)}$. In particular, $\overline{f(X)}\setminus f(X)\ne\emptyset$, which means that $f(X)$ is not closed in $Y$. Taking into account that the class $\C$ is stable under closures, we conclude that $\overline{f(X)}$ is a topologized semigroup in the class $\C$. So, we can replace $Y$ by $\overline{f(X)}$ and assume that the subgroup $f(X)$ is dense in $Y$. Replacing $Y$ by its isomorphic copy, we can assume that $X\cap Y=\emptyset$.  In this case we can consider the ideal sum $U_f(X,Y)$ and conclude that it belongs to the class $\C$ (since $\C$ is stable under ideal unions). By Theorem~\ref{t:2.2}(2), the topological group $X$ is not closed in $U_f(X,Y)$, which means that $X$
is not $\eC$-closed.
\smallskip

To prove the ``if'' part, assume that the Weil-complete topological group $X$ is not $\eC$-closed. Then $X$ admits a non-closed topological isomorphic embedding $f:X\to Y$ into a topologized semigroup $Y\in\C$. By Theorem~\ref{t:2.5n}(1), the complement $\overline{f(X)}\setminus f(X)$ is an ideal in $\overline{f(X)}$.
\end{proof}

\section{Topological groups of (pre)compact exponent}

In this section we study topological groups of compact exponent.
We shall say that a topological group $X$ has {\em compact exponent} (resp. {\em finite exponent\/}) if there exists a number $n\in\IN$ such that the set $\{x^n:x\in X\}$ is contained is a compact (resp. finite) subset of $X$. A complete topological group has compact exponent if and only if it has {\em precompact exponent} in the sense that for some $n\in\IN$ the set $nX$ is precompact. A subset $A$ of a topological group $X$ is called {\em precompact} if for any neighborhood $U\subset X$ of the unit there exists a finite subset $F\subset X$ such that $A\subset FU\cap UF$. 

\begin{lemma}\label{l:Prot} A topological group $X$ is precompact if and only if for any neighborhood $U\subset X$ of the unit there exists a finite subset $F\subset X$ such that $X=FUF$.
\end{lemma}

\begin{proof} The ``only if'' part is trivial. To prove the ``if'' part, assume that for any neighborhood $U\subset X$ of the unit there exists a finite subset $F\subset X$ such that $A\subset FUF$.
Given a neighborhood $W=W^{-1}\subset X$ of the unit, we need to find a finite subset $E\subset X$ such that $X=EW=WE$. Choose a neighborhood $U\subset X$ of the unit such that $UU^{-1}\subset W$. By our assumption, there exists a finite set $F\subset X$ such that $X=FUF$.
By \cite[12.6]{PB}, for some $x,y\in F$ there exists a finite subset $B\subset G$ such that $G=B(xUy)(xUy)^{-1}$. Then $G=BxUU^{-1}x^{-1}$ and hence $G=EW=WE^{-1}$ for $E=Bx$.
\end{proof}

The following proposition shows that our definition of a group of finite exponent is equivalent to the standard one.

\begin{proposition}\label{l:exp} A group $X$ has finite exponent if and only if there exists $n\in\IN$ such that for every $x\in X$ the power $x^n$ coincides with the unit of the group $X$.
\end{proposition}

\begin{proof} The ``if'' part is trivial. To prove the ``only if'' part, assume that $X$ has finite exponent and find $n\in\IN$ such that the set $F=\{x^n:x\in X\}$ is finite.

It follows that for every $x\in F$ the powers $x^{kn}$, $k\in\IN$, belong to the set $F$. So, by the Pigeonhole Principle, $x^{in}=x^{jn}$ for some numbers $i<j$. Consequently, for the number $m_x=j-i$ the power $x^{nm_x}$ is the unit $e$ of the group $X$. Then for the number $m=\prod_{x\in F}m_x$ we have $\{x^{n^2m}:x\in X\}\subset\{x^{nm}:x\in F\}=\{e\}$.
\end{proof}

This characterization implies that being of finite exponent is a 3-space property.

\begin{corollary} Let $H$ be a normal subgroup of a group $G$. The group $G$ has finite exponent if and only if $H$ and $G/H$ have finite exponent.
\end{corollary}

A similar 3-space property holds also for topological groups of compact exponent. A subgroup $H$ of a group $G$ is called {\em central} if $H$ is contained in the {\em center} $Z(G)=\{x\in G:\forall g\in G\;\;xg=gx\}$ of the group $G$.

\begin{proposition}\label{p:3exp} Let $Z$ be a closed central subgroup of a topological group $X$. The topological group $X$ has precompact exponent if and only if the topological groups $Z$ and $Y=X/Z$ have precompact exponent.
\end{proposition}

\begin{proof} If the topological group $X$ has precompact exponent, then for some $n\in\IN$ the set $nX=\{x^n:n\in X\}$ is precompact in $X$. It follows that the intersection $(nX)\cap Z\supset nZ$ is precompact in $Z$ and the image $q(nX)=nY$ of $nX$ under the quotient homomorphism $q:X\to X/Z=Y$ is precompact in the quotient topological group $Y$.
\smallskip

The proof of the ``if'' part is more complicated. Assume that the topological groups $Z$ and $Y:=X/Z$ have precompact exponent. Then there exist natural numbers $n$ and $m$ such that the sets $A:=nZ$ and $B:=mY$ are precompact. The set $B$ contains the unit of the group $Y$ and hence $B^k\subset B^n$ for all positive $k\le n$.

We claim that the set $nmX:=\{x^{nm}:x\in X\}$ is precompact in $X$. Given any open neighborhood $U=U^{-1}\subset G$ of the unit, we should find a finite subsets $F\subset G$ such that the set $nmX\subset FU\cap UF$. Since the set $nmX$ is symmetric, it suffices to find $F\subset X$ such that $nmX\subset FU$. Using the continuity of the group operations, choose a neighborhood $V\subset X$ of the unit such that $V=V^{-1}$ and $V^{3n+1}\subset U$. Let $q:X\to Y$ be the quotient homomorphism.

\begin{claim}\label{cl:w} For the precompact set $B^n$ the intersection $W=\bigcap_{x\in q^{-1}(B^n)}xV^3x^{-1}$ is a neighborhood of the unit in $X$.
\end{claim}

\begin{proof} By the precompactness of $B^n$ and the openness of the quotient homomorphism $q:X\to X/Z$, there exists a finite set $F\subset X$ such that $B^n\subset q(FV)$. We claim that the neighborhood $W'=\bigcap_{y\in F}yVy^{-1}$ is contained in $xV^3x^{-1}$ for any $x\in q^{-1}(B^n)$. Indeed, for any $w\in W'$ and $x\in q^{-1}(B^n)$, we can find $y\in F$ such that $q(x)\in q(yV)$ and hence $x=yvz$ for some $v\in V$ and $z\in Z$. Then $w\in yVy^{-1}\subset  xz^{-1}v^{-1}Vvzx^{-1}=xv^{-1}Vvx^{-1}\subset xV^3x^{-1}$ (here we use that $z$ belongs to the center of the group $X$).
\end{proof}

By the precompactness of the sets $A,B$, there exist a finite set $A'\subset A\subset Z$ such that $A\subset A'V$ and a finite set $B'\subset q^{-1}(B)$ such that $B\subset q(B')\cdot q(W)=q(B'W)$. We claim that the finite set $F=\{b^na:b\in B',\;a\in A'\}$ has the desired property: $x^{mn}\in FU$ for any $x\in X$.

The choice of $m$ ensures that $x^m\in q^{-1}(B)\subset q^{-1}(q(B'W))=B'WZ$. So, we can find elements $b\in B'$, $w\in W$ and $z\in Z$ such that $x^m=bwz$ and hence $x^{mn}=(bwz)^n=(bw)^nz^n\in (bw)^nA\subset (bw)^nA'V$ (we recall that the element $z\in Z$ belongs to the center of $X$).

Observe that $(bw)^n=\Big(\prod_{i=1}^nb^iwb^{-i}\Big)b^n$.
For every $i\le n$ the element $b^i$ belongs to $q^{-1}(B^n)$ and by Claim~\ref{cl:w}, $b^iwb^{-i}\in b^nV^3b^{-n}$. So, $$x^{mn}\in (bw)^nA'V=\Big(\prod_{i=1}^nb^iwb^{-i}\Big)b^nA'V\subset (b^nV^{3}b^{-n})^nb^nA'V=b^nV^{3n}A'V=b^nA'V^{3n+1}\subset FU.$$
\end{proof}

For complete topological groups, the precompactness of exponent is recognizable by separable subgroups.

\begin{proposition}\label{p:separ-exp} A complete topological group $X$ has precompact exponent if and only if each countable subgroup of $X$ has precompact exponent.
\end{proposition}

This proposition can be easily derived from the following (probably known) lemma.

\begin{lemma}\label{l:separ} A subset $A$ of a topological group $X$ is precompact if and only if for each countable subgroup $H\subset X$ the intersection $A\cap H$ is precompact in the topological group $H$.
\end{lemma}

\begin{proof} The ``only if'' part is trivial. To prove the ``only if'' part, assume that $A$ is not precompact. Then $A\cup A^{-1}$ is not precompact and hence there exists a neighborhood $U=U^{-1}\subset X$ of the unit such that $A\cup A^{-1}\ne FU$ for any finite subset $F\subset X$. By Zorn's Lemma, there exists a maximal subset $E\subset A\cup A^{-1}$ which is {\em $U$-separated} in the sense that $x\notin yU$ for any distinct points $x,y\in E$. The maximality of $E$ guarantees that for any $x\in A\cup A^{-1}$ there exists $y\in E$ such that $x\in yU$ or $y\in xU$ (and hence $x\in yU^{-1}=yU$). Consequently, $A\cup A^{-1}=EU$. The choice of $U$ ensures that the set $E$ is infinite. Then we can choose any infinite countable set $E_0\subset E$ and consider the countable subgroup $H$ generated by $E_0$. It follows that the intersection $H\cap (A\cup A^{-1})$ containes the infinite $U$-separated set $E_0$ and hence is not precompact in $H$.
\end{proof}

\section{On $\mathsf{e}{:}\C$-closed topological groups}\label{s:eC}

In this section we collect some results on $\mathsf{e}{:}\C$-closed topological groups for various classes $\C$.

First, observe that Theorems~\ref{t:Weil}, \ref{t:2.5n} and \ref{t:2.3} imply the following theorem (announced as Theorem~\ref{t1.1} in the introduction).

\begin{theorem}\label{t1.1r} A topological group $X$ is $\eTS$-closed if and only if $X$ is Weil-complete and for every continuous homomorphism $f:X\to Y$ into a Hausdorff topological semigroup $Y$ the complement $\overline{f(X)}\setminus f(X)$ is not an ideal in the semigroup $\overline{f(X)}$.
\end{theorem}

We recall that a topologized semigroup $X$ is defined to be a {\em powertopological semigroup} if it is semitopological and for every $n\in\IN$ the power map $X\to X$, $x\mapsto x^n$, is continuous. By $\pTS$ we denote the class of Hausdorff powertopological semigroups.


\begin{theorem}\label{t:exp} Each complete topological group $X$ of compact exponent is $\epTS$-closed.
\end{theorem}

\begin{proof} Fix a number $n\in\IN$ and a compact set $K\subset X$ such that $\{x^n:x\in X\}\subset K$. To show that $X$ is $\epTS$-closed, assume that $X$ is a subgroup of some Hausdorff powertopological semigroup $Y$. The Hausdorff property of $Y$ ensures that the compact set $K$ is closed in $Y$. Then the continuity of the power map $p:Y\to Y$, $p:y\mapsto y^n$, implies that the set $$\{y\in Y:y^n\in K\}$$ containing $X$ is closed in $Y$ and hence contains $\bar X$.
If $X$ is not closed in $Y$, then we can find a point $y\in \bar X\setminus X$ and conclude that $y^n\in K\subset X$. But this contradicts Theorem~\ref{t:2.5n}(2).
\end{proof}

\begin{corollary} For a topological group $X$ of precompact exponent the following conditions are equivalent:
\begin{enumerate}
\item $X$ is complete;
\item $X$ is $\eTG$-closed;
\item $X$ is $\eTS$-closed;
\item $X$ is $\epTS$-closed.
\end{enumerate}
\end{corollary}

\begin{proof} The implications $(4)\Ra(3)\Ra(2)$ follow from the inclusions $\pTS\supset\TS\supset\TG$, $(2)\Ra(1)$ and $(1)\Ra(4)$ follow from Theorems~\ref{t:Weil} and \ref{t:exp}, respectively.
\end{proof}




\begin{proposition}\label{p:eTS-center} If a topological group $X$ is $\eTS$-closed, then for any closed normal subgroup $N\subset X$ the center $Z(X/N)$ of the quotient group $X/N$ has precompact exponent.
\end{proposition}

\begin{proof} Let $G=X/N$ be the quotient topological group, $q:X\to G$ be the quotient homomorphism and $Z=\{z\in Z:\forall g\in G\;zg=gz\}$ be the center of the group $G$. Assuming that $Z$ does not have precompact exponent, we conclude that the completion $\bar Z$ of $Z$ does not have compact exponent. Applying Theorem~\ref{t1.1r}, we obtain a continuous injective homomorphism $f:\bar Z\to Y$ to a topological group $Y$ such that the closure $\overline{f(\bar Z)}=\overline{f(Z)}$ of $f(\bar Z)$ in $Y$ contains an element $y$ such that $y^n\notin f(\bar Z)$ for all $n\in\IN$.

Observe that the family $\tau_Z=\{Z\cap f^{-1}(U):U$ is open in $Y\}$ is a Hausdorff topology on $Z$ turning it into a topological group, which is topologically isomorphic to the topological group $f(Z)$. Then the completion $\bar Z$ of the topological group $(Z,\tau_Z)$ contains an element $z\in\bar Z$ such that $z^n\notin Z$ for all $n\in\IN$.

Let $\Tau_G$ be the topology of the topological group $G$. Taking into account that the subgroup $Z$ is central in $G$, we can show that the family $\tau_e=\{U\cdot V:e\in U\in\Tau_G,\;e\in V\in\tau_Z\}$ satisfies the Pontryagin Axioms \cite[1.3.12]{AT} and hence is a neighborhood base at the unit of some Hausdorff group topology $\tau$ on $G$. The definition of this topology implies that the subgroup $Z$ remains closed in the topology $\tau$ and the subspace topology $\{U\cap Z:U\in\tau\}$ on $Z$ coincides with the topology $\tau_Z$. Then the completion $\bar Z$ of the topological group $(Z,\tau_Z)$ is contained in the completion $\bar G$ of the topological group $(G,\tau)$ and hence $z\in\bar Z\subset\bar G$. Now consider the subsemigroup $S$ of $\bar G$, generated by the set $G\cup\{z\}$. Observe that $\{z^n\}_{n\in\w}\subset \bar Z\setminus Z=\bar Z\setminus G$. Since the group $Z$ is central in $G$, the element $z$ commutes with all elements of $G$. This implies that $S=\{gz^n:g\in G,\;n\in\w\}$ and hence $S\setminus G= \{gz^n:g\in G,\;n\in\IN\}$ is an ideal in $G$. Let $i:G\to S$ be the identity homomorphism. Then for the homomorphism $h=i\circ q:X\to S$ the complement $\overline{h(X)}\setminus h(X)=S\setminus G$ is an ideal in $S$. By Theorem~\ref{t1.1r}, the topological group $X$ is not $\eTS$-closed. This is a desired contradiction showing that the topological group $Z(G)=Z(X/N)$ has precompact exponent.
\end{proof}

We recall that for a topological group $X$ its central series $\{e\}=Z_0(X)\subset Z_1(X)\subset\cdots$ consists of the subgroups defined recursively as $Z_{n+1}(X)=\{z\in X:\forall x\in X\;\;zxz^{-1}x^{-1}\in Z_n(X)\}$ for $n\in\w$.

\begin{corollary}\label{c:eTS=>Zn} If a topological group $X$ is $\eTS$-closed, then for every $n\in\w$ the subgroup $Z_n(X)$ has compact exponent.
\end{corollary}

\begin{proof} First observe that the topological group $X$ is complete, being $\eTS$-closed. Then its closed subgroups $Z_n(X)$, $n\in\w$, also are complete. So, it suffices to prove that for every $n\in\w$ the topological group $Z_n(X)$ has precompact exponent. This will be proved by induction on $n$. For $n=0$ the trivial group $Z_0(X)=\{e\}$ obviously has precompact exponent. Assume that for some $n\in\w$ we have proved that the subgroup $Z_n(X)$ has precompact exponent. By Proposition~\ref{p:eTS-center}, the center $Z(X/Z_n(X))$ of the quotient topological group $X/Z_n(X)$ has precompact exponent. Since $Z(X/Z_n(X))=Z_{n+1}(X)/Z_n(X)$, we see that the quotient topological group $Z_{n+1}(X)/Z_n(X)$ has precompact exponent. By  Proposition~\ref{p:3exp}, the topological group $Z_{n+1}(X)$ has precompact exponent.  \end{proof}

Corollary~\ref{c:eTS=>Zn} implies the following characterization of $\eTS$-closed nilpotent topological groups (announced in the introduction as Theorem~\ref{t:eTS-nilpotent}).

\begin{theorem}\label{t:AbelC} For a nilpotent topological group $X$ the following conditions are equivalent:
\begin{enumerate}
\item $X$ is $\eTS$-closed;
\item $X$ is $\epTS$-closed;
\item $X$ is Weil-complete and has compact exponent;
\item $X$ is complete and has compact exponent.
\end{enumerate}
\end{theorem}

\begin{proof}  The implications $(2)\Ra(1)$ and $(3)\Ra(4)$ are trivial, and $(4)\Ra(2)$ was proved in Theorem~\ref{t:exp}. It remains to prove that $(1)\Ra(3)$. So, assume that the nilpotent topological group $X$ is $\eTS$-closed. By Theorem~\ref{t:Weil}, $X$ is Weil-complete. By Corollary~\ref{c:eTS=>Zn}, for every $n\in\w$ the subgroup $Z_n(X)$ has compact exponent. In particular, $X$ has compact exponent, being equal to $Z_n(X)$ for a sufficiently large number $n$.
\end{proof}

We do not know if Theorem~\ref{t:AbelC} remains true for hypercentral topological groups. We recall that a topological group $X$ is {\em hypercentral} if for each closed normal subgroup $H\subsetneq X$ the quotient group $X/H$ has non-trivial center. Each nilpotent topological group is hypercentral.

\begin{problem} Has each $\eTS$-closed hypercentral topological group compact exponent?
\end{problem}

The following characterization of compact topological groups shows that the $\epTS$-closedness of $X$ in Theorem~\ref{t:AbelC} cannot be replaced by the $\esTS$-closedness. The equivalence $(1)\Leftrightarrow(2)$ was proved by Gutik \cite{Gutik14}.

\begin{theorem}\label{t:gutik} For a topological group $X$ the following conditions are equivalent:
\begin{enumerate}
\item $X$ is compact;
\item $X$ is $\esTS$-closed;
\item $X$ is $\erTG$-closed.
\end{enumerate}
\end{theorem}

\begin{proof} The implication $(1)\Ra(2,3)$ is trivial.
\smallskip

To prove that $(2)\Ra(1)$, assume that a topological group $X$ is $\esTS$-closed. Then it is $\eTG$-closed and hence complete (by Theorem~\ref{t:Weil}). Assuming that $X$ is not compact, we conclude that $X$ is not totally bounded. So, there exists a neighborhood $V\subset X$ of the unit such that $X\not\subset FV\cup VF$ for any finite subset $F\subset X$.

Chose any element $0\notin X$ and consider the space $X_0=X\cup\{0\}$ endowed with the Hausdorff topology $\tau$ consisting of sets $W\subset  X_0$ such that $W\cap X$ is open in $X$ and if $0\in W$, then $X\setminus W\subset FV$ for some finite subset $F\subset X$. Extend the group operation of $X$ to a semigroup operation on $X_0$ letting $0x=0=x0$ for all $x\in X_0$. It is easy to see that $X_0$ is a Hausdorff semitopological semigroup containing $X$ as a non-closed subgroup and witnessing that $X$ is not $\esTS$-closed.
\smallskip

To prove that $(3)\Ra(1)$, assume that a topological group $X$ is $\erTG$-closed. Then it is $\eTG$-closed and hence complete (by Theorem~\ref{t:Weil}). Assuming that $X$ is not compact, we conclude that $X$ is not totally bounded. By Lemma~\ref{l:separ}, $X$ contains a countable subgroup which is not totally bounded. Now Lemma~\ref{l:erTG} (proved below) implies that $X$ is not $\erTG$-closed, which is a desired contradiction.
\end{proof}

A topology $\tau$ on a group $X$ is called {\em right-invariant} (resp. {\em shift-invariant}) if $\{Ux:U\in\tau,\;x\in X\}=\tau$ (resp. $\{xUy:U\in\tau,\;x,y\in X\}=\tau$). This is equivalent to saying that $(X,\tau)$ is a right-topological (resp. semitopological) group.

\begin{lemma}\label{l:erTG} If a topological group $X$ contains a countable subgroup $Z$ which is not totally bounded, then the group $\IZ\times X$ admits a Hausdorff right-invariant topology $\tau$ such that the subgroup $\{0\}\times X$ is not closed in the right-topological group $(\IZ\times X,\tau)$ and $\{0\}\times X$ is topologically isomorphic to $X$. Moreover, if the subgroup $Z$ is central in $X$, then $(\IZ\times X,\tau)$ is a semitopological group.
\end{lemma}

\begin{proof} Identify the product group $\IZ\times X$ with the direct sum $\IZ\oplus X$. In this case the group $X\subset \IZ\oplus X$ is identified with the subgroup $\{0\}\times X$ of the group $\IZ\times X$. Let $Z=\{z_k\}_{k\in\w}$ be an enumeration of the countable subgroup $Z$. Since $Z$ is not totally bounded, there exists a neighborhood $W=W^{-1}\subset X$ of the unit such that $Z\not\subset FW^3F$ for any finite subset $F\subset X$ (see Lemma~\ref{l:Prot}). Using this property of $Z$, we can inductively construct a sequence of points $(x_n)_{n\in\w}$ of $Z$ such that for every $n\in\w$ the following condition is satisfied:
\begin{itemize}
\item[(a)] $x_n\notin F_nW^3F_n^{-1}$ where
\item[(b)] $F_n=\{e\}\cup\big\{x_{i_1}x_{i_2}\cdots x_{i_k}z_j^{\e}:k\in\w,\;n>i_1>\dots>i_k,\;j\le n,\;\e\in\{0,1\}\big\}$.
\end{itemize}

For every $m\in\w$ consider the subset $$\Sigma_m:=\{(0,e)\}\cup\{(n,x_{i_1}\!\cdots x_{i_n}):n\in\IN,\;
i_1>\dots>i_n>m\}\subset \IZ\times X.$$

On the group $G:=\IZ\oplus X$,  consider the topology $\tau$ consisting of subsets $W\subset G$ such that for every $g\in W$ there exists $m\in\w$ and a neighborhood $U_g\subset X$ of $g$ such that $\Sigma_mU_g\subset W$.
The definition of the topology $\tau$ implies that for any $W\in\tau$ and $a\in G$ the set $Wa$ belongs to $\tau$. So, $(G,\tau)$ is a right-topological group. If the subgroup $Z$ is central, then for every $a\in G$ and $g\in aW$ we get $a^{-1}g\in W$, so we can find a neighborhood $U\subset X$ of $a^{-1}g$ and $m\in\w$ such that $\Sigma_mU\subset W$. Then $U_{g}:=aU$ is a neighborhood of $g$ in $X$ such that $\Sigma_mU_g=\Sigma_maU=a\Sigma_mU\subset aW$, which means that the set $aW$ belongs to the topology $\tau$ and the topology $\tau$ is invariant.

Let us show that for any open set $U\subset X$ and any $m\in\w$ the set $\Sigma_mU$ belongs to the topology $\tau$.

For every $g\in\Sigma_mU$ we can find $u\in U$ and a sequence $i_1>\cdots i_n>m$ such that $g=x_{i_1}\cdots x_{i_n}u$. Choose  neighborhoods $U_e,U_e'\subset X$ of the unit such that $uU_e\subset U$ and $U'_ex_{i_1}\cdots x_{i_n}u\subset x_{i_1}\cdots x_{i_n}uU_e$. Then
$$\Sigma_{i_1}U'_eg=\Sigma_{i_1}U_e'x_{i_1}\cdots x_{i_n}u\subset \Sigma_{i_1}x_{i_1}\cdots x_{i_n}uU_e\subset \Sigma_mU$$
and hence $\Sigma_mU\in\tau$.

Observe that for every $U\subset X$ and $m\in\w$, have $\Sigma_mU\cap X=U$, which implies that $X$ is a subgroup of the right-topological group $(\IZ\oplus X,\tau)$. The subgroup $X$ is not closed in $\IZ\oplus X$ as $\bar X$ contains any point $(n,x)\in\IZ\times X$ with $n\le 0$.

It remains to check that the right-topological semigroup $(G,\tau)$ is Hausdorff.
Given any element $g=(n,x)\in G\setminus\{(0,e)\}$, we should find a neighborhood $U_e\subset X$ and $m\in\w$ such that $\Sigma_mU_e\cap (\Sigma_mU_eg)=\emptyset$. If $x\notin \bar Z$, then we can find a neighborhood $U_e=U_e^{-1}\subset X$ of the unit such that
$U_exU_e\cap\bar Z=\emptyset$ and hence $\Sigma_0U_e\cap (\Sigma_0U_eg)\subset \IZ\times (ZU_e\cap ZU_ex)=\emptyset$.

So, we assume that $x\in\bar Z$ and hence $x\in z_mW$ for some $m\in\w$.
Choose a neighborhood $V=V^{-1}\subset W$ of the unit such that $Vz_m\subset z_mW$ and if $x\ne e$, then $x\notin V^2$.

We claim that $\Sigma_mV\cap \Sigma_mVg=\emptyset$. Assuming that this intersection is not empty, fix an element $y\in \Sigma_mV\cap\Sigma_mVg$. The inclusion $y\in\Sigma_mV$ implies that $y=(k,x_{i_1}\cdots x_{i_k}v)$ for some numbers $k\in\IN$, $i_1>\cdots >i_k>m$, and $v\in V$.
On the other hand, the inclusion $y\in \Sigma_mVg$ implies that $y=(l,x_{j_1}\cdots x_{j_l}u)g$ for some numbers $l\in\IN$ and $j_1>\dots>j_l>m$ and some $u\in V$. It follows that
\begin{equation}\label{eq:rTG}(k,x_{i_1}\cdots x_{i_k}v)=y=(l,x_{j_1}\cdots x_{j_l}u)\cdot(n,x)=
(l+n,x_{j_1}\cdots x_{j_l}ux).
\end{equation}

Let $\lambda$ be the largest number $\le1+\min\{k,l\}$ such that $i_p=j_p$ for all $1\le p<\lambda$. Three cases are possible.

1) $\lambda\le \min\{k,l\}$. In this case the numbers $i_\lambda$ and $j_\lambda$ are well-defined and distinct. The equality (\ref{eq:rTG}) implies $x_{i_\lambda}\cdots x_{i_k}v=x_{j_\lambda}\cdots x_{j_l}ux$.
If $i_\lambda>j_\lambda$, then
$$
\begin{aligned}
x_{i_\lambda}&=x_{j_\lambda}\cdots x_{j_l}uxv^{-1}(x_{i_{\lambda+1}}\cdots x_{i_k})^{-1}\subset
x_{j_\lambda}\cdots x_{j_l}uz_mWv^{-1}(x_{i_{\lambda+1}}\cdots x_{i_k})^{-1}\subset\\
&\subset x_{j_\lambda}\cdots x_{j_l}Vz_mWV^{-1}(x_{i_{\lambda+1}}\cdots x_{i_k})^{-1}\subset
x_{j_\lambda}\cdots x_{j_l}z_mWWV^{-1}(x_{i_{\lambda+1}}\cdots x_{i_k})^{-1}\subset F_{i_\lambda}W^3F_{i_\lambda},
\end{aligned}
$$
which contradicts the choice of $x_{i_\lambda}$.

If $i_\lambda<j_\lambda$, then
$$
\begin{aligned}
&x_{j_\lambda}=x_{i_\lambda}\cdots x_{i_k}vx^{-1}u^{-1}(x_{j_{\lambda+1}}\cdots x_{j_l})^{-1}\subset
x_{i_\lambda}\cdots x_{i_k}VW^{-1}z_m^{-1}V^{-1}(x_{j_{\lambda+1}}\cdots x_{j_l})^{-1}\subset\\
&\subset x_{i_\lambda}\cdots x_{i_k}VW^{-1}W^{-1}z_m^{-1}(x_{j_{\lambda+1}}\cdots x_{j_l})^{-1}\subset
x_{i_\lambda}\cdots x_{i_k}W^3(x_{j_{\lambda+1}}\cdots x_{j_l}z_m)^{-1}\subset F_{j_\lambda}W^3F^{-1}_{j_\lambda},
\end{aligned}
$$
which contradicts the choice of $x_{j_\lambda}$.
\smallskip

2) $\lambda=1+\min\{k,l\}$ and $k=l$. In this case the equation (\ref{eq:rTG}) implies that $n=0$ and $v=ux$. Then $e\ne x=vu^{-1}\in V^2$, which contradicts the choice of $V$.
\smallskip

3) $\lambda=1+\min\{k,l\}$ and $k<l$. In this case the equation (\ref{eq:rTG}) implies that
$v=x_{j_\lambda}\cdots x_{j_l}ux$ and hence $x_{j_\lambda}=vx^{-1}u^{-1}(x_{j_{\lambda+1}}\cdots x_{j_l})^{-1}\subset VW^{-1}z_{m}^{-1}V^{-1}(x_{j_{\lambda+1}}\cdots x_{j_l})^{-1}\subset
VW^{-1}W^{-1}z_{m}^{-1}(x_{j_{\lambda+1}}\cdots x_{j_l})^{-1}\subset
W^3(x_{j_{\lambda+1}}\cdots x_{j_l}z_m)^{-1}\subset F_{j_\lambda}W^3F_{j_\lambda}^{-1},$
which contradicts the choice of $x_{j_\lambda}$.

4) $\lambda=1+\min\{k,l\}$ and $l<k$. In this case the equation (\ref{eq:rTG}) implies that
$x_{i_\lambda}\cdots x_{i_k}v=ux$ and hence $x_{i_\lambda}=uxv^{-1}(x_{i_{\lambda+1}}\cdots x_{i_k})^{-1}\subset Vz_mWV^{-1}(x_{i_{\lambda+1}}\cdots x_{i_k})^{-1}\subset
z_mW^3(x_{i_{\lambda+1}}\cdots x_{i_l})^{-1}\subset
F_{i_\lambda}W^3F_{i_\lambda}^{-1},$
which contradicts the choice of $x_{i_\lambda}$. This contradiction finishes the proof of the Hausdorff property of the topology $\tau$.
\end{proof}

Lemmas~\ref{l:separ} and \ref{l:erTG} have two implications.

\begin{corollary}\label{c:esTG} Each $\esTG$-closed topological group has compact center.
\end{corollary}

\begin{corollary}\label{c:AesTG} An Abelian topological group is compact if and only if it is $\esTG$-closed.
\end{corollary}

\begin{problem} Is a topological group compact if it is $\esTG$-closed?
\end{problem}


\section{On ${\mathsf i}{:}\C$-closed topological groups}\label{s:iC}

In this section we collect some results on ${\mathsf i}{:}\C$-closed topological groups for various classes $\C$ of topologized semigroups. First we prove that for topological groups of precompact exponent, many of such closedness properties are equivalent.

\begin{theorem}\label{t:i-exp} For a topological group $X$ of precompact exponent the following conditions are equivalent:
\begin{enumerate}
\item $X$ is $\iTS$-closed;
\item $X$ is $\iTG$-closed.
\end{enumerate}
\end{theorem}

\begin{proof} The implications $(1)\Ra(2)$ is trivial and follows from the inclusion $\TG\subset\TS$. To prove that $(2)\Ra(1)$, assume that $X$ is $\iTG$-closed and take any continuous injective homomorphism $f:X\to Y$ to a Hausdorff topological semigroup $Y$.
We need to show that $f(X)$ is closed in $Y$. Replacing $Y$ by $\overline{f(X)}$, we can assume that the group $f(X)$ is dense in $Y$.
We claim that $Y$ is a topological group.

First observe that the image $e_Y=f(e_X)$ of the unit $e_X$ of the group $X$ is a two-sided unit of the semigroup $Y$ (since the set $\{y\in Y:ye_Y=y=ye_Y\}$ is closed in $Y$ and contains the dense subset $f(X)$~).

Since the complete group $X$ has precompact exponent, it has a compact exponent and hence for some number $n\in\IN$ of the set $nX=\{x^n:x\in X\}$ has compact closure $K:=\overline{nX}$. By the continuity of $h$ and the Hausdorff property of $Y$, the image $f(K)$ is a compact closed subset of $Y$. Consequently, the set $Y_n=\{y\in Y:y^n\in f(K)\}$ is closed in $Y$. Taking into account that $Y_n$ contains the dense subset $f(X)$, we conclude that $Y_n=Y$.

Now consider the compact subset $\Gamma=\{(x,y)\in f(K)\times f(K):xy=e_Y\}$ in $Y\times Y$. Let $\pr_1,\pr_2:\Gamma\to f(K)$ be the coordinate projections. We claim that these projections are bijective. Since $f(X)$ is a group, for every $z\in f(X)$ there exists a unique element $y\in f(X)$ with $zy=e_Y$. This implies that the projection $\pr_1:\Gamma\to f(K)$, $\pr_1:(z,y)\mapsto z$, is injective. Given any element $z\in f(K)$ find an element $x\in K\cap f^{-1}(z)$ and observe that $x^{-1}\in K^{-1}=K$ and hence the pair $(z,y):=(f(x),f(x^{-1}))$ belongs to $\Gamma$ witnessing that the map $\pr_1:\Gamma\to f(K)$ is surjective. Being a bijective continuous map defined on the compact space $\Gamma$, the map $\pr_1:\Gamma\to f(K)$ is a homeomorphism. By analogy we can prove that the projection $\pr_2:\Gamma\to f(K)$, $\pr_2:(z,y)\mapsto y$, is a homeomorphism.
Then the inversion map $i:f(K)\to f(K)$, $i=\pr_2\circ \pr_1^{-1}:f(K)\to f(K)$ is continuous.

Now consider the continuous map $\bar i:Y\to Y$ defined by $\bar i(y)=y^{n-1}\cdot i(y^n)$ for $y\in Y$. This map is well-defined since $y^n\in f(K)$ for all $y\in Y$. Observe that for every element $y$ of the group $f(X)$,  the element  $\bar i(y)=y^{n-1}\cdot i(y^n)=y^{n-1}(y^n)^{-1}=y^{-1}$ coincides with the inverse element of $y$ in the group $f(X)$. Consequently, $y\cdot \bar i(y)=e_Y=\bar i_Y(y)\cdot y$ for all $y\in f(X)$ and by the continuity of the map $\bar i$ this equality holds for every $y\in Y$. This means that each element $y$ of the semigroup $Y$ has inverse $\bar i(y)$ and hence $Y$ is a group. Moreover, the continuity of the map $\bar i$ ensures that $Y$ is a topological group. So, $f:X\to Y$ is an injective continuous homomorphism to a topological group. Since $X$ is $\iTG$-closed, the image $f(X)$ is closed in $Y$.
\end{proof}

Theorems~\ref{t:i-exp} and \ref{t:AbelC} imply the following characterization.

\begin{corollary}\label{c:iTS=eTS+iTG} A nilpotent topological group $X$ is $\iTS$-closed if and only if $X$ is $\eTS$-closed and  $\iTG$-closed.
\end{corollary}


The above results allow us to reduce the problem of detecting $\iTS$-closed topological groups to the problem of detecting $\iTG$-closed topological groups. So, now we establish some properties of $\iTG$-closed topological groups.

\begin{theorem}\label{t:iTG=>center} The center of any $\iTG$-closed topological group $X$ is compact.
\end{theorem}

\begin{proof} To derive a contradiction, assume that the center $Z$ of an $\iTG$-closed topological group $X$ is not compact. Being $\iTG$-closed, the topological group $X$ is complete and so is its closed subsemigroup $Z$. By Theorem~\ref{t:banakh}, the non-compact complete Abelian topological group $Z$ is not $\iTG$-closed and hence admits a non-complete weaker Hausdorff group topology $\tau_Z$.

Let $\Tau$ be the topology of $X$ and $\Tau_e=\{U\in\Tau:e\in U\}$.
Consider the family $$\tau_e=\{V\cdot U:V\in\Tau_e,\;e\in U\in\tau_Z\}$$ of open neighborhoods of the unit in the topological group $X$. It can be shown that $\tau_e$ satisfies the Pontryagin Axioms \cite[1.3.12]{AT} and hence is a base of some Hausdorff group topology $\tau\subset\Tau$ on $X$. Observe that the topology $\tau$ induces the topology $\tau_Z$ on the subgroup $Z$, which remains closed in the topology $\tau$. Since the topological group $(Z,\tau_Z)$ is not complete, the topological group $X_\tau=(X,\tau)$ is not complete, too.
Then the identity map $X\to \bar X_\tau$ into the completion $\bar X_\tau$ of $X_\tau$ has non-closed image, witnessing that the topological group $X$ is not $\iTG$-closed. This is a desired contradiction, completing the proof of the theorem.
\end{proof}

\begin{problem} Is a topological group compact if it is $\iqTG$-closed?
\end{problem}

Next, we study the interplay between the $\iTG$-closedness and minimality of topological groups.

We recall that a topological group $X$ is {\em minimal\/} if each continuous bijective  homomorphism $h:X\to Y$ to a topological group $Y$ is a topological isomorphism. This definition implies the following (trivial) characterization.

\begin{proposition}\label{p:min} A minimal topological group $X$ is $\iTG$-closed if and only if $X$ is $\eTG$-closed.
\end{proposition}

In Theorem~\ref{t:pseudo} we shall characterize $\iTG$-closed topological groups among $\w$-narrow topological groups of countable pseudocharacter. A topological space $X$ is defined to have {\em countable pseudocharacter} if for each point $x\in X$ there exists a countable family $\U$ of open sets in $X$ such that $\bigcap\U=\{x\}$.

We recall that a topological group $X$ is {\em $\w$-narrow} if for any neighborhood $U\subset X$ of the unit in $X$ there exists a countable set $C\subset X$ such that $X=CU$. The following classical theorem of Guran \cite{Guran} (see also \cite[Theorem 3.4.23]{AT}) describes the structure of $\w$-narrow topological groups.

\begin{theorem}[Guran]\label{t:Guran} A topological group $X$ is $\w$-narrow if and only if $X$ is topologically isomorphic to a subgroup of a Tychonoff product $\prod_{\alpha\in A}P_\alpha$ of Polish groups.
\end{theorem}

In the subsequent proofs we shall use the known Open Mapping Principle for Polish groups due to S.~Banach \cite{Banach} (see also \cite[9.10]{Ke}).

\begin{theorem}[Banach]\label{t:open} Each continuous surjective homomorphism $h:X\to Y$ between Polish groups is open.
\end{theorem}

Now we are able to present a characterization of $\iTG$-closed $\w$-narrow topological groups of countable pseudocharacter. By $\PG$ we denote the class of {\em Polish groups}, i.e. topological groups whose topological space is Polish (= separable completely metrizable).

\begin{theorem}\label{t:pseudo}  For an $\w$-narrow topological group $X$ of countable pseudocharacter the following conditions are equivalent:
\begin{enumerate}
\item $X$ is Polish and minimal;
\item $X$ is complete and minimal;
\item $X$ is $\iTG$-closed;
\item $X$ is $\iPG$-closed.
\end{enumerate}
\end{theorem}

\begin{proof} The implications $(1)\Ra(2)$ and $(3)\Ra(4)$ are trivial, and $(2)\Ra(3)$ follows from (the trivial) Proposition~\ref{p:min}. To prove that $(4)\Ra(1)$, assume that the topological group $X$ is $\iPG$-closed.

\begin{claim}\label{cl:topi} The topological group $X$ is Polish.
\end{claim}

\begin{proof} By Guran's Theorem~\ref{t:Guran}, the $\w$-narrow topological group $X$ can be identified with a subgroup of a Tychonoff product $\prod_{\alpha\in A}P_\alpha$ of Polish groups. For a subset $B\subset A$ by $p_B:X\to\prod_{\alpha\in B}P_\alpha$ we denote the coordinate projection.

Since $X$ has countable pseudocharacter, there exists a countable subfamily $B\subset A$ such that the coordinate projection $p_B:X\to\prod_{\beta\in B}P_\beta$ is injective. Taking into account that the topological group $X$ is $\iPG$-closed, we conclude that the image $p_B(X)$ in a closed subgroup of the Polish group $\prod_{\beta\in B}P_\beta$ and hence $p_B(X)$ is a Polish group.

We claim that the homomorphism $p_B:X\to p_B(X)$ is open. Assuming the opposite, we can find a neighborhood $U\subset X$ of the unit whose image $p_B(U)$ is not a neighborhood of the unit in the Polish group $p_B(X)$. Since $X\subset\prod_{\alpha\in A}P_\alpha$, there exists a finite subset $F\subset A$ and a neighborhood $U_F\subset \prod_{\alpha\in F}P_\alpha$ of the unit such that $p_F^{-1}(U_F)\subset U$.
Let $C:=B\cup F$ and consider the coordinate projection $p_C:X\to\prod_{\alpha\in C}P_\alpha$. Taking into account that the topological group $X$ is $\iPG$-closed, we conclude that the subgroup $p_C(X)$ is closed in the Polish group $\prod_{\alpha\in C}P_\alpha$ and hence $p_C(X)$ is Polish. Observe that the coordinate projection $p^C_B:p_C(X)\to p_B(X)$ is a continuous bijective homomorphism between Polish groups. By Theorem~\ref{t:open}, this homomorphism is open, which is not possible as
the image $p^C_B(p_C(U))=p_B(U)$ of the neighborhood $p_C(U)$ of the unit in $p_C(X)$ is not a neighborhood of the unit in the group $p_B(X)$. This contradiction shows that the continuous bijective homomorphism $p_B:X\to p_B(X)$ is open and hence is a topological isomorphism. Since the topological group $p_B(X)$ is Polish, so is the topological group $X$.
\end{proof}

It remains to show that the Polish group $X$ is minimal. Given a continuous bijective homomorphism  $h:X\to Z$ to a topological group $Z$, we need to check that $h$ is open. Observe that the topological group $Z$ is $\iPG$-closed, $\w$-narrow, and has countable pseudocharacter (being the continuous bijective image of the $\iPG$-closed Polish group $X$). By Claim~\ref{cl:topi}, the topological group $Z$ is Polish and by Theorem~\ref{t:open}, the homomorphism $h$ is open.
\end{proof}

The following (known) example shows that Theorem~\ref{t:pseudo} does not generalize to nilpotent topological groups.

\begin{example}\label{ex:HW}{\rm Consider the classical Weyl-Heisenberg group $H(w_0):=H(\IR)/Z$ where  $$H(\IR)=\left\{\left(\begin{array}{ccc}1&a&b\\0&1&c\\0&0&1\end{array}\right):a,b,c\in\IR\right\}
\mbox{ \ \ and \ \ } Z=\left\{\left(\begin{array}{ccc}1&0&b\\0&1&0\\0&0&1\end{array}\right):b\in\IZ\right\}. $$ The group $H(w_0)$ is nilpotent, Polish, locally compact, and minimal, see \cite[\S5]{DM}, \cite[5.5]{DM10}, \cite{Meg}. Being minimal and non-compact, the complete group $H(w_0)$ is not MAP. Being minimal and complete, the Weyl-Heisenberg group $H(w_0)$ is $\iTG$-closed. On the other hand, $H(w_0)$ admits a continuous homomorphism onto the real line, which implies that $H(w_0)$ is not $\hTG$-closed.}
\end{example}

\section{On ${\mathsf h}{:}\C$-closed topological groups}\label{s:hC}

In this section we collect some results on ${\mathsf h}{:}\C$-closed topological groups for various classes $\C$ of topologized semigroups.
First observe the following trivial characterization of $\hC$-closed topological groups.

\begin{proposition}\label{p:h=i+q} Let $\C$ be a class of topological groups. A topological group $X$ is $\hC$-closed if and only if for any closed normal subgroup $N\subset X$ the quotient topological group $X/N$ is $\iC$-closed.
\end{proposition}

Proposition~\ref{p:h=i+q} and Theorem~\ref{t:i-exp} implies the following characterization.

\begin{theorem}\label{t:exp=>hTS=hTG} For a topological group $X$ of precompact exponent the following conditions are equivalent:
\begin{enumerate}
\item $X$ is $\hTS$-closed;
\item $X$ is $\hTG$-closed.
\end{enumerate}
\end{theorem}

Also Theorem~\ref{t:iTG=>center} and Proposition~\ref{p:h=i+q} imply:

\begin{corollary}\label{c:hTS-center} For any closed normal subgroup $N$ of an $\hTG$-closed topological group $X$ the quotient topological group $X/N$ has compact center $Z(X/N)$. In particular, the Abelization $X/\overline{[X,X]}$ of $X$ is compact.
\end{corollary}

We recall that for a topological group $X$ its central series $\{e\}=Z_0(X)\subset Z_1(X)\subset\cdots$ consists of the subgroups defined recursively as $Z_{n+1}(G)=\{z\in X:\forall x\in X\;\;zxz^{-1}x^{-1}\in Z_n(X)\}$ for $n\in\w$.

\begin{proposition}\label{p:hTG=>Z_n} If a topological group $X$ is $\hTG$-closed, then for every $n\in\w$ the subgroup $Z_n(X)$ is compact.
\end{proposition}

\begin{proof} The compactness of the subgroups $Z_n(X)$ will be proved by induction on $n$. For $n=0$ the compactness of the trivial group $Z_0(X)=\{e\}$ is obvious. Assume that for some $n\in\w$ we have proved that the subgroup $Z_n(X)$ is compact. By Corollary~\ref{c:hTS-center}, the center $Z(X/Z_n(X))$ of the quotient topological group $X/Z_n(X)$ is compact. Since $Z(X/Z_n(X))=Z_{n+1}(X)/Z_n(X)$, we see that the quotient topological group $Z_{n+1}(X)/Z_n(X)$ is compact. By \cite[1.5.8]{AT} (the 3-space property of the compactness), the topological group $Z_{n+1}(X)$ is compact. \end{proof}

Proposition~\ref{p:hTG=>Z_n} implies the following characterization of $\hTG$-closed nilpotent topological groups, first proved by Dikranjan and Uspenskij \cite[3.9]{DU}.

\begin{corollary}[Dikranjan, Uspenskij] A nilpotent topological group is compact if and only if it is $\hTG$-closed.
\end{corollary}


We recall that a topological group $X$ is {\em totally minimal\/} if for any closed normal subgroup $N\subset X$ the quotient topological group $X/N$ is minimal. Proposition~\ref{p:h=i+q} and Theorem~\ref{t:pseudo} imply the following characterization.

\begin{corollary}\label{c:ana}  For an $\w$-narrow topological group $X$ of countable pseudocharacter the following conditions are equivalent:
\begin{enumerate}
\item $X$ is $\hTG$-closed;
\item $X$ is $\hPG$-closed.
\item $X$ is Polish and totally minimal.
\end{enumerate}
\end{corollary}

In fact, the countable pseudocharacter can be removed from this corollary.
Following \cite{Luk}, we define a topological group $X$ to be {\em totally complete}  if for any closed normal subgroup $N\subset X$ the quotient topological group $X/N$ is complete. It is easy to see that each $\hTG$-closed topological group is totally complete. By \cite[3.45]{Luk}, a topological group is $\hTG$-closed if it is totally complete and totally minimal. These observations are complemented by the following characterization.

\begin{theorem}\label{t:absw} For an $\w$-narrow topological group $X$ the following conditions are equivalent:
\begin{enumerate}
\item $X$ is $\hTG$-closed;
\item $X$ is totally complete and totally minimal.
\end{enumerate}
\end{theorem}

\begin{proof} The implication $(2)\Ra(1)$ was proved by Luk\'acs  \cite[3.45]{Luk}. The implication $(1)\Ra(2)$ will be derived from the following lemma.

\begin{lemma}\label{l:minimal} Each $\w$-narrow $\hTG$-closed topological group $X$ is minimal.
\end{lemma}

\begin{proof} We should prove that each continuous bijective homomorphism $f:X\to Y$ to a topological group $Y$ has continuous inverse $f^{-1}:Y\to X$. Since $X$ embeds into the Tychonoff product of Polish groups, it suffices to check that for every continuous homomorphism $p:X\to P$ to a Polish group $P$ the composition $p\circ f^{-1}:Y\to P$ is continuous. Since $X$ is $\hTG$-closed, the image $p(X)$ is closed in $P$. Replacing $P$ by $p(X)$, we can assume that $P=p(X)$. Let $K_X=p^{-1}(1)$ be the kernel of the homomorphism $p$. Observe that the quotient topological group $X/K_X$ admits a bijective continuous homomorphism $\bar p:X/K_X\to P$ such that $p=\bar p\circ q_X$ where $q_X:X\to X/K_X$ is the quotient homomorphism.  It follows that the quotient topological group $X/K_X$ has countable pseudocharacter. Moreover, the topological group $X/K_X$ is $\w$-narrow and $\hTG$-closed being a continuous homomorphic image of the $\w$-narrow $\hTG$-closed topological group $X$. By Corollary~\ref{c:ana}, the group $X/K_X$ is Polish and minimal. Consequently, $\bar p:X/K_X\to P$ is a topological isomorphism.

Now we shall prove that the image $K_Y=f(K_X)$ of $K_X$ is closed in $Y$.
Consider the continuous homomorphism $fp:X\to Y\times P$, $fp:x\mapsto (f(x),p(x))$, and observe that its image $fp(X)$ is a closed subgroup of $Y\times P$ by the $\iTG$-closedness of $X$. Consequently, the homomorphism $p\circ f^{-1}$ has closed graph $\Gamma=\{(y,p\circ f^{-1}(y)):y\in Y\}=\{(f(x),p(x)):x\in X\}=fp(X)$. Since the intersection $\Gamma\cap(Y\times \{1\})$ is a closed subset of $Y\times\{1\}$ the homomorphism $p\circ f^{-1}$ has closed kernel $$K_Y=\{y\in Y:p\circ f^{-1}(y)=1\}=\{y\in Y:(y,1)\in\Gamma\},$$ which coincides with $f(K_X)$. Let $Y/K_Y$ be the quotient topological group and $q_Y:Y\to Y/K_Y$ be the quotient homomorphism.

The continuous bijective homomorphism $f:X\to Y$ induces a continuous bijective homomorphism $\hat f:X/K_X\to Y/K_Y$ making the following diagram commutative.
$$
\xymatrix{
&X\ar_p[dl]\ar^f[r]\ar^{q_X}[d]&Y\ar^{q_Y}[d]\\
P&X/K_X\ar^{\bar p}[l]\ar_{\hat f}[r]&Y/K_Y
}
$$
The minimality of the Polish group $X/K_X$ guarantees that the bijective homomorphism $\hat f$ is a topological isomorphism, which implies that the homomorphism $\bar p\circ \hat f^{-1}\circ q_Y=p\circ f^{-1}:Y\to P$ is continuous.
\end{proof}

Now we are able to prove the implication $(1)\Ra(2)$ of Theorem~\ref{t:absw}. Given an $\w$-narrow $\hTG$-closed topological group $X$ and a closed normal subgroup $N\subset X$, we should check that the quotient topological group $X/N$ is complete and minimal. Observe that $X/N$ is $\w$-narrow and $\hTG$-closed (being a continuous homomorphic image of the $\w$-bounded $\hTG$-closed topological group $X$). By Theorem~\ref{t:Weil}, the $\hTG$-closed topological group $X/N$ is complete and by Lemma~\ref{l:minimal}, it is minimal.
\end{proof}

\section{On $\pC$-closed topological groups}

In this section we collect some results on ${\mathsf p}{:}\C$-closed topological groups for various classes $\C$ of topologized semigroups.

\begin{theorem}\label{t:pC} Let $\C$ be a class of topologized semigroups, containing all Abelian topological groups. For a topological group $X$ the following conditions are equivalent:
\begin{enumerate}
\item $X$ is $\pC$-closed;
\item each closed subsemigroup of $X$ is $\hC$-closed;
\item each closed subgroup of $X$ is $\hC$-closed.
\end{enumerate}
\end{theorem}

\begin{proof} The equivalence $(1)\Leftrightarrow(2)$ follows immediately from the definitions of $\pC$-closed and $\hC$-closed topological groups, and $(2)\Ra(3)$ is trivial.

To prove that $(3)\Ra(2)$, assume that each closed subgroup of $X$ is $\hC$-closed, and take any closed subsemigroup $S\subset X$ of $X$. We claim that $S$ is a subgroup of $X$. Given any element $x\in S$, consider the closure $Z$ of the cyclic subgroup $\{x^n\}_{n\in\IZ}$. By our assumption, $Z$ is $\hC$-closed and by Theorem~\ref{t:banakh}, the Abelian $\hC$-closed topological group $Z$ is compact. Being a compact monothetic group, $Z$ coincides with the closure of the subsemigroup $\{x^n\}_{n\in\IN}$ and hence is contained in the closed subsemigroup $S$. Then $x^{-1}\in Z\subset S$ and $S$ is a subgroup of $X$. By (3), $S$  is $\hC$-closed.
\end{proof}

For topological groups of precompact exponent, Theorem~\ref{t:exp=>hTS=hTG} implies the following characterization.

\begin{corollary} For any topological group $X$ of precompact exponent the following conditions are equivalent:
\begin{enumerate}
\item $X$ is $\phTG$-closed;
\item $X$ is $\phTS$-closed.
\end{enumerate}
\end{corollary}



An interesting property of $\phTG$-closed topological groups was discovered by Dikranjan and Uspenskij \cite[3.10]{DU}. This property concerns the transfinite derived series $(X^{[\alpha]})_{\alpha}$ of a topological group $X$, which consists of the closed normal subgroups $X^{[\alpha]}$ of $X$, defined by the transfinite recursive formulas:
$$
\begin{aligned}
&X^{[0]}:=X,\\
&X^{[\alpha+1]}:=\overline{[X^{[\alpha]},X^{[\alpha]}]}\mbox{ for each ordinal $\alpha$},\\
&X^{[\gamma]}:=\bigcap_{\alpha<\gamma}X^{[\alpha]}\mbox{ for each limit ordinal $\gamma$}.
\end{aligned}
$$
The closed subgroup $$X^{[\infty]}:=\bigcap_{\alpha}X^{[\alpha]}$$ is called the {\em hypocommutator} of $X$. A topological group is called {\em hypoabelian} it its hypocommutator is trivial.

\begin{theorem}[Dikranjan, Uspenskij]\label{t:DU-hypocommutator} For each $\phTG$-closed topological group $X$ the quotient topological group $X/X^{[\infty]}$ is compact.
\end{theorem}

\begin{corollary}\label{c:hypoabelian} A hypoabelian topological group is compact if and only if it is $\phTG$-closed.
\end{corollary}

Since solvable topological groups are hypoabelian, Corollary~\ref{c:hypoabelian} implies the following characterization of compactness of solvable topological groups.

\begin{corollary} A solvable topological group $X$ is compact if and only if it is $\phTG$-closed.
\end{corollary}

By \cite[2.16]{DU}, {\em a balanced topological group is $\phTG$-closed if and only if it is $\cTG$-closed}.

\begin{problem} Is each balanced $\phTS$-closed topological group $\cTS$-closed?
\end{problem}

\section{On closedness properties of MAP topological groups}\label{s:Bohr}

In this section we establish some properties of MAP topological groups.  We recall that a topological group $X$ is {\em maximally almost periodic} (briefly, {\em MAP}) if it admits a continuous injective homomorphism into a compact topological group.

\begin{theorem} A topological group $X$ is compact if and only if $X$ is $\phTG$-closed and MAP.
\end{theorem}

\begin{proof} The ``only if'' part is trivial. To prove the ``if' part, assume that $X$ is $\phTG$-closed and MAP. Then $X$ is complete. Assuming that $X$ is not compact, we can apply Lemma~\ref{l:separ} and find a non-compact closed separable subgroup $H\subset X$. Since $X$ is $\phTG$-closed, the closed subgroup $H$ of $X$ is $\hTG$-closed and by Lemma~\ref{l:minimal}, it is minimal and being complete and MAP, is compact. But this contradicts the choice of $H$.
\end{proof}

The notion of a MAP topological group can be generalized as follows.
Let $\K$ be a class of topologized semigroups. A topologized semigroup $X$ is defined to be {\em $\K$-MAP} if it admits a continuous injective homomorphism $f:X\to K$ to some compact topologized semigroup $K\in\K$. So, MAP is equivalent to $\TG$-MAP.

\begin{theorem}\label{t:MAP->i=h} Let $\C$, $\K$ be two classes of Hausdorff topologized semigroups such that for any $C\in\C$ and $K\in\K$ the product $C\times K$ belongs to the class $\C$. A $\K$-MAP topologized semigroup $X$ is ${\mathsf i}{:}\C$-closed if and only if it is $\mathsf{h}{:}\C$-closed.
\end{theorem}

\begin{proof} The ``if'' part is trivial. To prove the ``only if'' part, assume that a topologized group $X$ is $\K$-MAP and $\iC$-closed. To prove that $X$ is  $\hC$-closed, take any continuous homomorphism $f:X\to Y$ to a topologized semigroup $Y\in\C$. Since $X$ is $\K$-MAP, there exists a continuous injective homomorphism $h:X\to K$ into a compact topologized semigroup $K\in\K$.
By our assumption, the topologized semigroup $Y\times K$ belongs to the class $\C$. Since the homomorphism $fh:X\to Y\times K$, $fh:x\mapsto (f(x),h(x))$, is continuous and injective, the image $\Gamma=fh(X)$ of the $\iC$-closed semigroup $X$ in the semigroup $Y\times K\in\C$ is closed.
By \cite[3.7.1]{Eng}, the projection $\pr:Y\times K\to Y$ is a closed map (because of the compactness of $K$). Then the projection $\pr(fh(X))=h(X)$ is closed in $Y$.
\end{proof}

Since compact topological groups are balanced, each MAP topological group admits a continuous injective homomorphism into a balanced topological group. We recall \cite[p.69]{AT} that a topological group $X$ is {\em balanced} if each neighborhood $U\subset X$ of the unit contains a neighborhood $V\subset X$ of the unit which is {\em invariant} in the sense that $V=xVx^{-1}$ for all $x\in X$.

\begin{proposition}\label{p:bal} Let $X$ be an $\iTG$-closed topological group. For each continuous injective homomorphism $f:X\to Y$ to a balanced topological group $Y$ and each closed normal subgroup $Z\subset X$ the image $f(Z)$ is closed in $Y$.
\end{proposition}

\begin{proof} To derive a contradiction, assume that the image $f(Z)$ of some closed normal subgroup $Z\subset X$ is not closed in $Y$. Let $\mathcal B_X$ be the family of open neighborhoods of the unit in the topological group $X$ and $\mathcal B_Y$ be the family of open invariant neighborhoods of the unit in the balanced topological group $Y$. It can be shown that the family $$\mathcal B=\{V\cdot (f|Z)^{-1}(W):V\in \mathcal B_X,\;\;W\in\mathcal B_Y\}$$satisfies the Pontryagin Axioms \cite[1.3.12]{AT} and hence is a base of some Hausdorff group topology $\tau$ on $X$. In this topology the subgroup $Z$ is closed and is topologically isomorphic to the topological group $f(Z)$ which is not closed in $Y$ and hence is not complete. Then the topological group $X_\tau=(X,\tau)$ is not complete too, and hence is not closed in its completion $\bar X_\tau$. Now we see that the identity homomorphism $\id:X\to X_\tau\subsetneq\bar X\tau$ witnesses that $X$ is not $\iTG$-closed. This contradiction completes the proof.
\end{proof}

In the proof of our next result, we shall need the (known) generalization of the Open Mapping Principle \ref{t:open} to homomorphisms between K-analytic topological groups.

 We recall \cite{Ro} that a topological space $X$ is {\em $K$-analytic} if $X=f(Z)$ for some continuous function $f:Z\to X$ defined on an $F_{\sigma\delta}$-subset $Z$ of a compact Hausdorff space $C$. It is clear that the continuous image of a $K$-analytic space is $K$-analytic and the product $A\times C$ of a $K$-analytic space $A$ and a compact Hausdorff space $C$ is $K$-analytic.

A topological group is called {\em $K$-analytic} if its topological space is $K$-analytic. In \cite[\S I.2.10]{Ro}) it was shown that  Banach's Open Mapping Principle~\ref{t:open}  generalizes to homomorphisms defined on $K$-analytic groups.

\begin{theorem}[K-analytic Open Mapping Principle]\label{t:openK} Each continuous homomorphism $h:X\to Y$ from a $K$-analytic topological group $X$ onto a Baire topological group $Y$ is open.
\end{theorem}

Theorem~\ref{t:openK} will be used in the proof of the following lemma.

\begin{lemma}\label{t:MAP-Ka} Let $X$ be an $\iTG$-closed MAP topological group. For a closed normal subgroup $Z\subset X$ and a continuous homomorphism $h:X\to Y$ to a topological group $Y$,  the image $h(Z)$ is compact if and only if $h(Z)$ is contained in a $K$-analytic subspace of\/ $Y$.
\end{lemma}

\begin{proof} The ``only if'' part is trivial. To prove the ``if'' part, assume that the image $h(Z)$ is contained in a $K$-analytic subspace $A$ of $Y$.

Being MAP, the group $X$ admits a continuous injective homomorphism $f:X\to K$ to a compact topological group $K$.
By Proposition~\ref{p:bal}, the image $f(Z)$ is a compact subgroup of $K$.

Now consider the continuous injective homomorphism $fh:X\to K\times Y$, $fh:x\mapsto (f(x),h(x))$. By the $\iTG$-closedness, the image $\Gamma:=fh(X)$ is closed in $K\times Y$. Then the space $H=\Gamma\cap (f(Z)\times A)$ is $K$-analytic (as a closed subspace of the $K$-analytic space $K\times A$). Observe that $H=\{(f(z),h(z):z\in Z\}$ is a subgroup of $X\times Y$.

By the K-analytic Open Mapping Principle~\ref{t:openK}, the bijective continuous homomorphism $\pr_K:H\to f(Z)$, $\pr_K:(x,y)\mapsto x$, from the $K$-analytic group $H$ to the compact group $f(Z)$ is open and hence is a topological isomorphism. Consequently, the topological group $H$ is compact and so is its projection $h(Z)$ onto $Y$.
\end{proof}

We recall that a topological group is {\em $\w$-balanced} if it embeds into a Tychonoff product of metrizable topological groups.

\begin{corollary}\label{c:Bohr2} If an $\w$-balanced MAP topological group $X$ is $\iTG$-closed, then each $\w$-narrow closed normal subgroup of $X$ is compact.
\end{corollary}

\begin{proof} Being $\w$-balanced and complete, the topological group $X$ can be identified with a closed subgroup of a Tychonoff product $\prod_{\alpha\in A}M_\alpha$ of complete metrizable topological groups.

Fix an $\w$-narrow closed normal subgroup $H$ in $X$ and observe that
for every $\alpha\in A$ the projection $\pr_\alpha(H)\subset M_\alpha$ is an $\w$-narrow and hence separable subgroup of the metrizable  topological group $M_\alpha$. Then the closure of $\pr_\alpha(H)$ in the complete metrizable topological group $M_\alpha$ is a Polish (and hence $K$-analytic) group. By Proposition~\ref{p:bal}, the group $\pr_\alpha(H)$ is compact. Being a closed subgroup of the product $\prod_{\alpha\in A}\pr_\alpha(H)$ of compact topological groups, the topological group $H$ is compact, too.
\end{proof}

We recall that for a topological group $X$ the {\em $\w$-conjucenter} $Z^\w(X)$ of $X$ consists of the points $z\in X$ whose conjugacy class $C_X(z):=\{xzx^{-1}:x\in X\}$ is $\w$-narrow in $X$.
A subset $A$ of a topological group $X$ is called {\em $\w$-narrow} if for each neighborhood $U\subset X$ of the unit there exists a countable set $B\subset X$ such that $A\subset BU\cap UB$.

\begin{theorem}\label{t:conjucent} Each $\iTG$-closed $\w$-balanced MAP topological group $X$ has compact $\w$-conjucenter $Z^\w(X)$.
\end{theorem}

\begin{proof} First we prove that $Z^\w(X)$ is precompact. Assuming the opposite, we can apply Lemma~\ref{l:separ} and find a countable subgroup $D\subset Z^\w(X)$ whose closure $\bar D$ is not compact. By the definition of $Z^\w(X)$, each element $x\in D$ has $\w$-narrow conjugacy class $C_X(x)$. By \cite[5.1.19]{AT}, the $\w$-narrow set $\bigcup_{x\in D}C_X(x)$ generates an $\w$-narrow subgroup $H$. It is clear that $H$ is normal. By Corollary~\ref{c:Bohr2}, the closure $\bar H$ of $H$ is compact, which is not possible as $\bar H$ contains the non-compact subgroup $\bar D$. This contradiction completes the proof of the precompactness of $Z^\w(X)$. Then the closure $\bar Z^\w(X)$ of the subgroup $Z^\w(X)$ in $X$ is a compact normal subgroup of $X$. The normality of $\bar Z^\w(X)$ guarantees that for every $z\in \bar Z^\w(X)$ the conjugacy class $C_X(z)\subset \bar Z^\w(X)$ is precompact and hence $\w$-narrow, which means that $z\in Z^\w(X)$. Therefore, the $\w$-conjucenter $Z^\w(X)=\bar Z^\w(X)$ is compact.
\end{proof}

For any topological group $X$ let us define an increasing transfinite sequence $(Z_\alpha(X))_{\alpha}$ of closed normal subgroups defined by the recursive formulas
$$
\begin{aligned}
&Z_0(X)=\{e\},\\
&Z_{\alpha+1}(X):=\{z\in X:\forall x\in X\;xzxz^{-1}\in Z_\alpha(X)\}\mbox{ \  for any ordinal $\alpha$ and}\\
&Z_\beta(X)\mbox{ is the closure of the normal group $\bigcup_{\alpha<\beta}Z_\beta(X)$ for a limit ordinal $\beta$}.
\end{aligned}
$$
The closed normal subgroup $Z_\infty(X)=\bigcup_{\alpha}Z_\alpha(X)$ is called the {\em hypercenter} of the topological group $X$. We recall that a topological group $X$ is {\em hypercentral} if for every closed normal subgroup $N\subsetneq X$ the quotient topological group $X/N$ has non-trivial center $Z(X/N)$. It is easy to see that a topological group $X$ is hypercentral if and only if its hypercenter equals $X$.
It follows that a discrete topological group is hypercentral if and only if it is hypercentral in the standard algebraic sense \cite[364]{Rob}.
Observe that each nilpotent topological group is hypercentral. More precisely, a group $X$ is nilpotent if and only if $Z_n(X)=X$ for some finite number $n\in\IN$.

\begin{corollary}\label{c:hypercenter} If an $\w$-balanced MAP topological group $X$ is $\iTG$-closed, then its hypercenter $Z_\infty(X)$ is contained in the $\w$-conjucenter $Z^\w(X)$ and hence is compact.
\end{corollary}

\begin{proof} By Theorem~\ref{t:conjucent}, the $\w$-conjucenter $Z^\w(X)$ of $X$ is compact. By transfinite induction we shall prove that for every ordinal $\alpha$ the subgroup $Z_\alpha(X)$ is contained in $Z^\w(X)$.
This is trivial for $\alpha=0$. Assume that for some ordinal $\alpha$ we have proved that $\bigcup_{\beta<\alpha}Z_\beta(X)\subset Z^\w(X)$.
If the ordinal $\alpha$ is limit, then
$$Z_\alpha(X)=\overline{\bigcup_{\beta<\alpha}Z_\beta(X)}\subset Z^\w(X).$$
Next, assume that $\alpha=\beta+1$ is a successor ordinal. To prove that $Z_\alpha(X)\subset Z^\w(X)$, take any point $z\in Z_\alpha(X)$ and observe that $C_X(z)=\{xzx^{-1}:x\in X\}\subset Z_\beta(X)\subset Z^\w(X)$ and hence $C_X(z)$ is $\w$-narrow. So, $z\in Z^\w(X)$ and $Z_\alpha(X)\subset Z^\w(X)$.

By the Principle of Tranfinite Induction, the subgroup $Z_\alpha(X)\subset Z^\w(X)$ for every ordinal $\alpha$. Then the hypercenter $Z_\infty(X)$ is contained in $Z^\w(X)$ and is compact,  being equal to $Z_\alpha(X)$ for a sufficiently large ordinal $\alpha$.
\end{proof}

\begin{corollary}\label{c:hypercentral} A hypercentral topological group $X$ is compact if and only if $X$ is $\iTG$-closed, $\w$-balanced, and MAP.
\end{corollary}

\begin{problem} Is each $\hTG$-closed hypercentral MAP topological group compact?
\end{problem}

\section{The compactness of $\hTG$-closed MAP-solvable topological groups}\label{s:MAP-solv}

In this section we detect compact topological groups among $\hTG$-closed MAP-solvable topological groups.
We define a topological group $X$ to be {\em MAP-solvable} if there exists a decreasing sequence $X=X_0\supset X_1\supset\dots \supset X_m=\{e\}$ of closed normal subgroups in $X$ such that for every $n<m$ the quotient group $X_n/X_{n+1}$ is Abelian and MAP. It is clear that each MAP-solvable topological group is solvable. By the Peter-Weyl Theorem~\cite[9.7.5]{AT} (see also \cite{Mor}), each locally compact Abelian group is MAP. This observation implies the following characterization.

\begin{proposition}\label{p:Bohr-s} A locally compact topological group is MAP-solvable if and only if it is solvable.
\end{proposition}

Now we prove that balanced MAP-solvable $\hTG$-closed topological groups are compact. We recall that a topological group $X$ is called a {\em balanced} if each neighborhood $U\subset X$ of the unit contains a neighborhood $V\subset X$ of the unit such that $xVx^{-1}\subset U$ for all $x\in X$.

\begin{theorem}\label{t:solvable} For a solvable topological group $X$ the following conditions are equivalent:
\begin{enumerate}
\item $X$ is compact;
\item $X$ is balanced, locally compact and  $\hTG$-closed;
\item $X$ is balanced, MAP-solvable and  $\hTG$-closed.
\end{enumerate}
\end{theorem}

\begin{proof} The implication $(1)\Ra(2)$ is trivial and $(2)\Ra(3)$ follows from Proposition~\ref{p:Bohr-s}.
To prove that $(3)\Ra(1)$, assume that a topological group $X$ balanced, MAP-solvable and $\hTG$-closed. Being MAP-solvable, $X$ admits a decreasing sequence of closed normal subgroups $X=X_0\supset X_2\supset \cdots \supset X_m=\{e\}$ such that for every  $n<m$ the quotient group $X_n/X_{n+1}$ is Abelian and MAP.

To prove that the group $X=X/X_m$ is compact, it suffices to show that for every $n\le m$ the quotient group $X/X_n$ is compact. This is trivial for $n=0$. Suppose that for some $n<m$ the group $X/X_n$ is compact. Assuming that the quotient group $G:=X/X_{n+1}$ is not compact, we conclude that the normal Abelian subgroup $A:=X_n/X_{n+1}$ of $G$ is not compact. The $\hTG$-closedness of $X$ implies the  $\hTG$-closedness of the quotient group $G$. Then $G$ is complete and by Lemma~\ref{l:separ}, the non-compact closed subgroup $A$ of $X$ contains a countable subgroup $Z$ whose closure $\bar Z$ is not compact.

\begin{claim} For every $a\in A$ its conjugacy class $C_G(a)=\{xax^{-1}:x\in G\}$ is compact.
\end{claim}

\begin{proof} Consider the continuous map $f:G\to A$, $f:x\mapsto xax^{-1}$, and observe that it is constant on cosets $xA$, $x\in G$. Consequently, there exists a function $\tilde f:G/A\to A$ such that $f=\tilde f\circ q$ where $q:G\to G/A$ is a quotient homomorphism.
Since the group $G/A$ carries the quotient topology, the continuity of $f$ implies the continuity of $\tilde f$. Now the compactness of the quotient group $G/A=X/X_n$ implies that the set $\tilde f(G/A)=f(G)=C_G(a)$ is compact.
\end{proof}

It follows that the union $\bigcup_{z\in Z}C_G(z)$ is $\sigma$-compact and hence generates a $\sigma$-compact subgroup $H\subset A$, which is normal in $G$. Then its closure $\bar H$ is a normal closed $\w$-narrow subgroup in $G$.

\begin{claim} The topological group $G$ is balanced and MAP.
\end{claim}

\begin{proof} The topological group $G$ is balanced, being a quotient group of the balanced topological group $X$. Let $\mathcal B_G$ be the family of all open invariant neighborhoods of the unit in the balanced group $G$. To prove that $X$ is MAP, we shall use the fact that the Abelian topological group $A$ is MAP, which implies that the strongest totally bounded group topology $\tau_A$ on $A$ is Hausdorff. Let us observe that for every $x\in G$ the continuous automorphism $A\to A$, $a\mapsto xax^{-1}$, remains continuous in the topology $\tau_A$. Using this fact, it can be shown that the family
$$\mathcal B=\{UV:U\in\mathcal B_G,\;V\in\tau_A\}$$ satisfies the Pontryagin Axioms \cite[1.3.12]{AT} and hence is a base of some Hausdorff group topology $\tau$ on $G$. Observe that the subgroup $A$ of $(G,\tau)$ is precompact and has compact quotient group $(G,\tau)/A$. Since the precompactness is a 3-space property (see \cite[1.5.8]{AT}), the topological group $(G,\tau)$ is precompact and its Raikov-completion $\bar G$ is compact. The identity homomorphism $\id:G\to(G,\tau)\subset\bar G$ witnesses that the topological group $G$ is MAP.
\end{proof}

Since the topological group $G$ is balanced, MAP and $\hTG$-closed, we can apply Corollary~\ref{c:Bohr2} and conclude that the $\w$-narrow closed normal subgroup $H$ of $G$ is compact, which is not possible as $H$ contains a closed non-compact subgroup $\bar Z$. This contradiction completes the proof of the compactness of the subgroup $A=X_n/X_{n+1}$. Now the compactness of the groups $X/X_n$ and $A=X_n/X_{n+1}$ imply the compactness of the quotient group $X/X_{n+1}$, see \cite[1.5.8]{AT}.
\end{proof}

Since each discrete topological group is  locally compact and balanced, Theorem~\ref{t:solvable} implies

\begin{corollary} A solvable topological group is finite if and only if it is discrete and $\hTG$-closed.
\end{corollary}

\section{Some counterexamples}\label{s:Iso}

In this section we collect some counterexamples.

Our first example shows that Theorem~\ref{t:AbelC} does not generalize to solvable (even meta-Abelian) discrete groups. A group $G$ is called {\em meta-Abelian} if it contains a normal Abelian subgroup $H$ with Abelian quotient $G/H$.

For an Abelian group $X$ let $X\rtimes C_2$ be the product $X\times \{-1,1\}$ endowed with the operation $(x,u)*(y,v)=(xy^u,uv)$ for $(x,u),(y,v)\in X\times \{-1,1\}$. The semidirect product $X\rtimes C_2$ is meta-Abelian (since $X\times\{1\}$ is a normal Abelian subgroup of index 2 in $X\rtimes C_2$).

By $\TSone$ we  denote the family of topological semigroups $X$ satisfying the separation axiom $T_{1}$ (which means that finite subsets in $X$ are closed). Since $\TS\subset\TSone$, each $\TSone$-closed topological semigroup is $\TS$-closed.

\begin{proposition}\label{p:sd} For any Abelian group $X$ the semidirect product $X\rtimes C_2$ endowed with the discrete topology is an   ${\mathsf e}{:}\!\TSone$-closed MAP topological group.
\end{proposition}

\begin{proof} First we show that the group $X\rtimes C_2$ is MAP. By Peter-Weyl Theorem \cite[9.7.5]{AT}, the Abelian discrete topological group $X$ is MAP and hence admits an injective homomorphism $\delta:X\to K$ to a compact Abelian topological group $K$. It easy to see that the semidirect product $K\rtimes C_2$ endowed with the group operation $(x,u)*(y,v)=(xy^u,uv)$ for $(x,u),(y,v)\in K\times C_2$ is a compact topological group and the map $\delta_2:X\rtimes C_2\to K\rtimes C_2$, $\delta_2:(x,u)\mapsto (\delta(x),u)$, is an injective homomorphism witnessing that the discrete topological group $X\rtimes C_2$ is MAP.
\smallskip

To show that $X\rtimes C_2$ is ${\mathsf e}{:}\!\TSone$-closed, fix a topological semigroup $Y\in\TSone$ containing the group $G:=X\rtimes C_2$ as a discrete subsemigroup. Identify $X$ with the normal  subgroup $X\times\{1\}$ of $G$. First we show that  $X$ is closed in $Y$. Assuming the opposite, we can find an element $\hat y\in \bar X\setminus X$.  Consider the element $p:=(e,-1)\in X\rtimes C_2$ and observe that for any element $x\in X$ we get $pxp=(e,-1)(x,1)(e,-1)=(x^{-1},1)=x^{-1}$ and hence $pxpx=e$, where $e$ is the unit of the groups $X\subset G$. The closedness of the singleton $\{e\}$ in $Y$ and the continuity of the multiplication in the semigroup $Y$ guarantee that the set $$Z=\{y\in Y:ypyp=e\}$$ is closed in $Y$ and hence contains the closure of the group $X$ in $Y$. In particular, $\hat y\in \bar X\subset\bar Z=Z$. So, $\hat yp\hat yp=e$. Since the subgroup $G$ of $Y$ is discrete, there exists a neighborhood $V_e\subset Y$ of $e$ such that $V_e\cap G=\{e\}$. By the continuity of the semigroup operation on $Y$, the point $\hat y$ has a neighborhood $W\subset Y$ such that $WpWp\subset V_e$. Fix any element $x\in W\cap X$ and observe that $(W\cap X)pxp\subset X\cap (WpWp)\subset X\cap V_e=\{e\}$, which is not possible as the set $W\cap X$ is infinite and so is its right shift $(W\cap X)pxp$ in the group $G$. This contradiction shows that the set $X$ is closed in $Y$.

Next, we show that the shift $Xp=X\times\{-1\}$ of the set $X$ is closed in $Y$. Assuming that $Xp$ has an accumulating point $y^*$ in $Y$, we conclude that $y^*p$ is an accumulating point of the group $X$, which is not possible as $X$ is closed in $Y$. So,  the sets $X$ and $Xp$ are closed in $Y$ and so is their union $G=X\cup Xp$, witnessing that the group $G$ is ${\mathsf e}{:}\!\TSone$-closed.
\end{proof}

Since the isometry group $\Iso(\IZ)$ of $\IZ$ is isomorphic to the semidirect product $\IZ\rtimes\C_2$,
Proposition~\ref{p:sd} implies the following fact.

\begin{example}\label{e:Iso} The group $\Iso(\IZ)$ endowed with the discrete topology is ${\mathsf e}{:}\!\TSone$-closed and MAP but does not have compact exponent.
\end{example}

By Dikranjan-Uspenskij Theorem~\ref{t:DU-nilpotent}, each $\hTG$-closed nilpotent topological group is compact.
Our next example shows that this theorem does not generalize to solvable topological groups and thus resolves in negative Question~3.13 in \cite{DU} and Question 36 in \cite{DSh}.

\begin{example}\label{ex:solv} The Lie group
$$\Iso_+(\IC)=\left\{\left(\begin{array}{cc}a&b\\0&1\end{array}\right):a,b\in\IC,\;|a|=1\right\}$$
of orientation-preserving isometries of the complex plane is  $\hTS$-closed and MAP-solvable but not compact and not MAP.
\end{example}

\begin{proof} The group $\Iso_+(\IC)$ is topologically isomorphic to the semidirect product $\IC\rtimes\IT$ of the additive group $\IC$ of complex numbers and the multiplicative group $\IT=\{z\in\IC:|z|=1\}$.
The semidirect product is endowed with the group operation $(x,a)\cdot(y,b)=(x+ay,ab)$ for $(x,a),(y,b)\in\IC\rtimes \IT$.

It is clear that the group $G:=\IC\rtimes\IT$ is meta-Abelian (and hence solvable) and not compact. Being solvable and locally compact, the group $G$ is MAP-solvable (see Proposition~\ref{p:Bohr-s}). To prove that $G$ is $\hTS$-closed, take any continuous homomorphism $h:G\to Y$ to a Hausdorff topological semigroup $Y$ and assume that the image $h(G)$ is not closed in $Y$. Replacing $Y$ by $\overline{h(G)}$, we can assume that the subgroup $h(G)$ is dense in the topological semigroup $Y$.

\begin{claim} The homomorphism $h$ is injective.
\end{claim}

\begin{proof} Consider the closed normal subgroup $H=h^{-1}(h(0,1))$ of $G$, the quotient topological group $G/H$ and the quotient homomorphism $q:G\to G/H$. It follows that $h=\tilde h\circ q$ for some continuous homomorphism $\tilde h:G/H\to Y$. We claim that the subgroup $H$ is trivial. To derive a contradiction, assume that $H$ contains some element $(x,a)\ne (0,1)$. Then for every $(y,b)\in \IC\rtimes \IT$ the normal subgroup $H$ of $G$ contains also the element $(y,b)(x,a)(y,b)^{-1}=(y+bx,ba)(-b^{-1}y,b^{-1})=(y+bx-ay,a)$ and hence contains the coset $\IC\times\{a\}$. Being a subgroup, $H$ also contains the set $(\IC\times\{a\})\cdot(\IC\times\{a\})^{-1}=\IC\times\{1\}$. Taking into account that the quotient group $G/(\IC\times\{1\})$ is compact, we conclude that $G/H$ is compact too. Consequently, $h(G)=\tilde h(G/H)$ is compact and closed in the Hausdorff space $Y$, which contradicts our assumption. This contradiction shows that the subgroup $H$ is trivial and the homomorphism $h$ is injective.
\end{proof}

\begin{claim} The map $h:G\to Y$ is a topological embedding.
\end{claim}

\begin{proof} Since $h$ is injective, the family $\tau=\{h^{-1}(U):U$ is open in $Y\}$ is a Hausdorff topology turning $G$ into a paratopological group. We need to show that $\tau$ coincides with the standard locally compact topology of the group $G$. Since the topology $\tau$ is weaker than the original product topology of $G=\IC\rtimes\IT$, the compact set $\{1\}\times\IT$ remains compact in the topology $\tau$. Then we can find a neighborhood $U\in\tau$ of the unit $(0,1)\in G$ such that $UU$ is disjoint with the compact set $\{1\}\times\IT$.

Using the compactness of the set $\{0\}\times \IT\subset G$ and the continuity of the multiplication in $G$, find a neighborhood $V\in\tau$ of the unit such that $\{ava^{-1}:v\in V,\;a\in\{0\}\times \IT\}\subset U$. We claim that $V\subset\{(z,a)\in\IC\rtimes\IT:|z|\le 1\}$.
Assuming the opposite, we could find an element $(v,a)\in V$ with $|v|>1$.

Since $|v^{-1}|<1$ and $|a|=1$, there are two complex numbers $b,c\in\IT$ such that $b+ac=v^{-1}\in\IC$. Observe that $(bv,a)=(0,b)(v,a)(0,b^{-1})\in U$ and similarly $(cv,a)\in U$. Then $$UU\ni(bv,a)(cv,a)=(bv+acv,a^2)=(1,a^2)\in \{1\}\times\IT,$$
which contradicts the choice of the neighborhood $U$. This contradiction shows that $V\subset\{(z,a)\in G:|z|\le1\}$ and hence $V$ has compact closure in the spaces $G$ and $(G,\tau)$. This means that the paratopological group $(G,\tau)$ is locally compact and, by the Ellis Theorem~\cite[2.3.12]{AT}, is a topological group. By the Open Mapping Principle~\ref{t:openK}, the identity homomorphism $G\to (G,\tau)$ is a topological isomorphism and so is the homomorphism $h:G\to h(G)\subset Y$. \end{proof}

Since the topological group $G$ is Weil-complete (being locally compact), Theorem~\ref{t:2.5n} guarantees that $Y\setminus h(G)$ is an ideal of the semigroup $Y$. Choose any element $y\in Y\setminus h(G)$ and observe that for the compact subset $K=h(\{0\}\times\IT)\subset h(G)\subset Y$ the compact set $yKy$ is contained in the ideal $Y\setminus h(G)$ and hence does not intersect $K$. By the Hausdorff property of the topological semigroup $Y$ and the compactness of $K$, the point $y$ has a neighborhood $V\subset Y$ such that $(VKV)\cap K=\emptyset$. Now take any element $v\in V\cap h(G)$ and find an element $(x,a)\in G$ with  $h(x,a)=v$. Let $b=-a^{-1}\in\IT$ and observe that $(x,a)(0,b)(x,a)=(x,ab)(x,a)=(x,-1)(x,a)=(0,-a)$. Then for the element $k=h(0,b)\in K$ we get $VKV\ni vkv=h(0,-a)\in K$, which contradicts the choice of the neighborhood $V$. This contradiction completes the proof of the $\hTS$-closedness of the Lie group $\IC\rtimes \IT$.

By Theorem~\ref{c:Bohr2}, the non-compact $\w$-narrow $\hTG$-closed group $\IC\rtimes\IT$ is not MAP.
\end{proof}

The following striking example of Klyachko, Olshanskii and Osin \cite[2.5]{KOO} shows that the $\phTS$-closedness does not imply compactness even for 2-generated discrete topological groups. A group is called {\em $2$-generated} if it is generated by two elements.

\begin{example}[Klyachko, Olshanskii, Osin]\label{ex:non-top} There exists a $\phTG$-closed infinite simple 2-generated discrete topological group $G$ of finite exponent. {\rm By Theorems~\ref{t:exp=>hTS=hTG} and~\ref{t:cTG}}, $G$ is $\phTS$-closed and is $\cTG$-closed.
\end{example}

We do not know if the groups constructed by Klyachko, Olshanskii and Osin can be $\cTS$-closed. So, we ask

\begin{problem} Is each $\cTS$-closed topological group compact?
\end{problem}

Finally, we present an example showing that an  $\hTG$-closed topological group needs not be $\eTS$-closed.

\begin{example} The symmetric group $\Sym(\w)$ endowed with the topology of pointwise convergence has the following properties:
\begin{enumerate}
\item $X$ is Polish, minimal, and not compact;
\item $X$ is complete but not Weil-complete;
\item $X$ is  $\hTG$-closed;
\item $X$ is not $\eTS$-closed.
\end{enumerate}
\end{example}

\begin{proof} The group $\Sym(\w)$ is Polish, being a $G_\delta$-subset of the Polish space $\w^\w$. The minimality of the group $\Sym(\w)$ is a classical result of Gaughan \cite{Gau}. Being Polish, the topological group $\Sym(\w)$ is complete. By \cite[7.1.3]{DPS}, the topological group $\Sym(\w)$ is not Weil-complete. By Theorem~\ref{t:Weil}, the topological group $\Sym(\w)$ is not $\eTS$-closed.

It remains to show that the topological group $X=\Sym(\w)$ is $\hTG$-closed. Let $h:X\to Y$ be a continuous homomorphism to a topological group $Y$. By \cite[7.1.2]{DPS}, the group $\Sym(\w)$ is topologically simple, which implies that the kernel $H=h^{-1}(1)$ of the homomorphism $h$ is either trivial or coincides with $X$. In the second case the group $h(X)$ is trivial and hence closed in $Y$. In the first case, the homomorphism $h$ is injective. By the minimality of $X$, the homomorphism $h$ is an isomorphic topological embedding. The completeness of $X$ ensures that the image $h(X)$ is closed in $Y$.
\end{proof}

\section*{Acknowledgements}

The author expresses his thanks to Serhi\u\i\ Bardyla and Alex Ravsky for fruitful discussions on the topic of the paper, to Dikran Dikranjan for helpful comments on the initial version of the paper, and to Michael Megrelishvili for a valuable remark on the minimality of the Weyl-Heisenberg group.

\end{document}